\documentclass[11pt,twoside]{article}
\usepackage{graphicx}
\usepackage{bm}
\usepackage{amssymb,amsmath}
\usepackage[T1]{fontenc}

\newcommand{\R}{\mathbb{R}}

\newtheorem{proposition}{Proposition}[section]

\newtheorem{lemma}[proposition]{Lemma}
\newtheorem{theorem}[proposition]{Theorem}
\newtheorem{remark}[proposition]{Remark}

\newcommand{\END}{\hfill$\Box$}
\newenvironment{proof}{{\it Proof}. }{\hfill\END\\[0.5ex]}

\newcommand{\Si}{\mathop{\rm Si}}
\newcommand{\Ci}{\mathop{\rm Ci}}

\numberwithin{equation}{section}

\begin{document}
\title{ Filon-Clenshaw-Curtis rules for a class of highly-oscillatory integrals with
logarithmic singularities}

\author{V. Dom\'{\i}nguez\thanks{
Dep. Ingenier\'{\i}a Matem\'{a}tica e Inform\'{a}tica,
E.T.S.I.I.T.
Universidad P\'{u}blica de Navarra.
Campus de Tudela
31500 - Tudela (SPAIN), {email}:{\tt victor.dominguez@unavarra.es}} }
\maketitle

\begin{abstract}
In this work we propose and analyse a numerical method for computing a family of
highly oscillatory integrals with logarithmic singularities. For these 
quadrature rules we derive error estimates in terms of $N$, the number of nodes,
$k$ the rate of
oscillations and  a Sobolev-like regularity of the function. We prove 
that the method is not only robust but the error even decreases, for fixed $N$, 
as
$k$ increases. Practical issues
about the implementation of the rule are also covered in this paper by: (a)
writing down ready-to-implement algorithms; (b) analysing the numerical
stability of the computations and (c) estimating the overall computational 
cost. We finish by showing some
numerical experiments which illustrate the theoretical results presented in
this paper.
\end{abstract}


\section{Introduction}
This paper  concerns itself with the approximation of 
\begin{equation}
\label{eq:defI}
{\cal I}^\alpha_k(f):= \int_{-1}^1  f(x)\,\log((x-\alpha)^2)\,{\exp(ik
x})\,{\rm d}x
\end{equation}
where $\alpha\in[-1,1]$. For the sake of simplicity, we will assume 
throughout this paper  that $k\ge 0$, although the algorithm and the theoretical
results can be straightforwardly adapted for
$k\le 0$. 

Our aim is to design numerical methods whose
rates of convergence do not depend on $k$ but only on $f$ and the number of
nodes of the quadrature rules. No information  about the derivatives, which
is very common in the approximation of oscillatory integrals (see
\cite{HuVa:06} or \cite{IsNo:05} and references therein), will be used, which
results in a simpler and less restrictive method. At first
sight,  $\alpha\in\{-1,0,1\}$ could be the more common cases  but
since the analysis we  develop here is actually valid for any
$\alpha\in[-1,1]$, we cover the general case in this paper.  

We choose in this work the Clenshaw-Curtis approach:
\begin{equation}
\label{eq:the:rule:k}
{\cal I}^\alpha_{k,N}(f):= \int_{-1}^1  {\cal Q}_N
f(x)\,\log((x-\alpha)^2)\,{\exp(ik x})\,{\rm d}x\approx {\cal I}^\alpha_{k}(f),\
\end{equation}
where  
\begin{equation}
 \label{eq:QN}
 \mathbb{P}_N\ni {\cal Q}_N f, \quad\text{s.t.}\quad \big( {\cal Q}_N
f\big)(\cos(n\pi/N))=f(\cos(n\pi/N)),\qquad n=0,\ldots,N. 
\end{equation}
In other words, ${\cal Q}_N f$ is the  polynomial of degree $N$ which
interpolates  $f$
at Chebyshev nodes. 

Classical, and modified, Clenshaw-Curtis rules 
\eqref{eq:the:rule:k} enjoy  very good properties which have made them very
popular in the scientific literature cf.
\cite{ClCu:60, Pi:2000, Quadpack:1983, SlSm:1980,MR821902} and been 
considered competitive  with respect to Gaussian rules even for smooth 
integrands (we refer to 
\cite{MR2403058} for an interesting discussion about this fact).
First of 
all, the error of the rule is, in the worst case, like the error of the 
interpolating polynomial in $L_1(-1,1)$. Thus, the rule is robust
respect to $k$ and it inherits the excellent approximation properties of the 
interpolant. On the other hand, and from a more practical view,
nested  grids can be used in the computations. Hence, if ${\cal 
I}_{k,N}(f)$ has been already computed, $ {\cal I}_{k,2N}(f)$  only requires 
$N$ new evaluations of $f$, i.e. previous calculations can be reused.
Moreover, by comparing both
approximations, a-posteriori error estimate is at our disposal almost for free.
Finally, ${\cal Q}_Nf$ can be expressed in the Chebyshev basis very fast,  in 
about ${\cal O}(N\log N)$ operations, using FFT techniques.

If $k=0$, or if $k$ is small enough ($k\le 2$ has been used throughout  this
paper), the
complex exponential can be incorporated to the definition of $f$. This leads us
to consider, 
in the same spirit, the following integral and numerical approximation,
\begin{equation}
\label{eq:the:rule:k:k0}
{\cal I}_{0}^\alpha(f):= \int_{-1}^1  f(x)\,\log((x-\alpha)^2)\,{\rm
d}x\approx \int_{-1}^1  {\cal Q}_Nf(x)\,\log((x-\alpha)^2)\,{\rm
d}x=:
{\cal I}^\alpha_{0,N}(f). 
\end{equation}
This problem is also dealt with in this work since the combination of both
algorithms gives rise to a  method which can be applied to non-, mildly
and highly oscillatory integrals.  

For these rule we will show that the rule converges superalgebraically for 
smooth functions $f$. Moreover, the error is not only not deteriorated as $k$ 
increases but it even decreases as $k^{-1}$ as $k\to \infty$. Furthermore, for 
some particular values of $\alpha$, which include the more common 
choices $\alpha\in \{-1,0,1\}$,  the error decay faster, as $k^{-2}$,  which 
means that both, the absolute and relative error of the rule decreases )cf. 
Theorem \ref{theo:Conv}).

The implementation of the rule hinges on finding a way to compute, fast and
accurately, the weights 
\begin{eqnarray}
\label{eq:omegank}
 \xi_n^\alpha(k)&:=&\int_{-1}^1 T_n(x)\log((x-\alpha)^2)\,\exp(ikx)\,{\rm
d}x,\quad k>2,\\
\xi_n^\alpha&:=&\xi^\alpha_n(0)=\int_{-1}^1 T_n(x)\,\log((x-\alpha)^2)\,{\rm d}x
\label{eq:omegank:k0}
\end{eqnarray}
($T_n(x):=\cos(n\arccos x)$ is the Chebyshev polynomial of the first kind)
for $n=0,1,\ldots,N$.
The second set of coefficients $(\xi^\alpha_n)_n$ is computed  by using  a
three-term 
recurrence relation which we show to be stable. For the first
set, $(\xi^\alpha_n(k))_n$, the situation is  more delicate. First we derive a
new three-term linear recurrence which can be used to evaluate
$\xi^\alpha_n(k)$. The calculations, however, turn out to be stable only for
$n\le k$. This could be understood, somehow, as  consequence of 
potentially handling  two  different  sources of  oscillations in 
$\xi^\alpha_n(k)$. The most obvious is  that coming from the complex
exponential, which is fixed independent of $n$. However, when $n$ is large, the 
Chebyshev polynomials, like the classical orthogonal polynomials, have all their
roots  in $[-1,1]$. This results in a increasing oscillatory behaviour of the
polynomial as $n\to\infty$. As long as the first oscillations source dominates
the second one, i.e.  as $k>n$, the recurrence is stable: any perturbation
introduced in the  computation is amplified very little. However, when 
$n>k$ increases,  such perturbations are hugely magnified,  which  
makes this approach completely useless. Of course, if $k$ is large,  so
should be $N$ to find these instabilities. Hence, this
only causes difficulties for practical computations  in the middle range, that
is, when $k$ is not yet very large but we need to use a relatively large number
of points to evaluate the integral within the prescribed precision.

This phenomenon is not new: It has been
already observed, among other examples, when computing  the simpler integral
\[
 \int_{-1}^1 T_n(x)\exp(ikx)\,{\rm d}x.
\]
(See  \cite{DoGrSm:2010} and references therein). Actually, the  problem is
circumvented using the same idea, the so-called Oliver method (cf.
\cite{Oliver1966}) which consists in rewriting appropriately the
difference equation  used before now as a
tridiagonal linear system whose (unique) solution gives the sought coefficients
except the last one which is part now of the right-hand-side.  Therefore, the 
evaluation
of this last coefficient has to be carried out in a different way. Thus, we
make use of an
asymptotic argument, namely the Jacobi-Anger expansion, which expresses
$\xi_N^\alpha(k)$  as a series whose terms are a product of Bessel functions and
integrals as in \eqref{eq:omegank:k0}. Despite the fact that it could
seem at first sight, the
series can also be summed in about ${\cal O}(N)$ operations. The resulting
algorithm has a cost ${\cal O}(N\log N)$, cost which
is lead by the FFT method used in the construction of the interpolant ${\cal
Q}_Nf$.

Let us point out that the case of $\alpha=0$, for both the oscillatory and
non-oscillatory case, has been previously considered  in \cite{BrHa:2007}  using
a different strategy. Roughly speaking, it relies on using the asymptotic
Jacobi-Anger expansion for all the coefficients, no matter how large  $k$ is
respect to $n$.  Our approach is, in our opinion, more optimal since the
algorithm is simpler to implement and the computational cost is smaller.

The interest in designing efficient methods for  approximating 
oscillatory integrals has been increased
in the last years,  fueled by new problems like high  frequency  scattering 
simulations  cf. \cite{BrHa:2007,
HuVa:07, DoGrSm:2010}.  For instance, in the boundary integral method, the
assembly of the matrix of the systems requires computing highly oscillatory 
integrals which are
smooth
except on the diagonal.  Hence,
after appropriate change of variables,  we can reduce the problem to evaluate
\[
\int_0^{1} f(s)\exp(ik s)\,{\rm d}s.
\]	
Typically, $f$ is smooth except at the end-points where an integrable
singularity, which could be either in the original integral or introduced in the
change of variables, occurs. Actually, the log-singularity is very common since
one can find it in the fundamental solutions for many differential operators in
2D, for instance, in the Helmholtz equation. 

Different strategies have been suggested for computing oscillatory integrals.
For instance steepest descent methods, based on analytic continuation in the
complex plane \cite{HuVa:06} or Levin methods which reduces the problem to
solving ODE by collocation methods \cite{Lev:97,Ol:2010}. On the other hand, we 
find
Filon rules which consists in interpolating the function by a (piecewise)
polynomial. Therefore, our method can be characterised as a Filon rule. The
general case for smooth functions has been considered eg. in
\cite{Is:04,Is:05,IsNo:05,Me:09,Xi:2007a,Xi:2007b}. Provided that the new
integral with the interpolating polynomial replacing the original function can 
be
computed exactly, a robust method is obtained in the sense that it converges as
the
size of the subintervals shrink to zero. Depending on the choice of the nodes we
have Filon-Clenshaw-Curtis rules,   Filon-Gaussian rules or,
if the derivatives, usually at the end points, are also interpolated,
Filon-Hermite rules. Oscillatory integrals with algebraic singularities in the 
integrand, and more general oscillators,  have been considered in 
\cite{MoXi:2011,KaXi:2013}. A different approach was considered in 
\cite{1207.2283} where the use of graded meshes toward the singularities 
has shown to be also efficient. Let us 
point out that this last example gives another example of the importance of 
having robust methods, which covers all possible values of $k$, since graded 
meshes can easily have very small subintervals so that the oscillations 
are reduced or even disappear. We would like to finish this introduction by
mentioning two recent works on this topic. 
First, in 
\cite{Sha:2013} a robust Matlab implementation of Filon-Gauss-Legendre rules is presented. That is, 
quadrature rules which
are based on integrating the interpolating polynomial on the Gauss-Legendre nodes. On the other hand, 
and for 3D geometries, we cite 
\cite{DoGa:2013} for a recent attempt to extend  Filon-Clenshaw-Curtis rules to computing 
highly oscillatory integrals
on the unit sphere.

This paper is structured as it follows: In section 2 we derive the error 
estimates
for the quadrature rule. In section 3 we deduce the algorithms to evaluate
the coefficients  \eqref{eq:omegank} and \eqref{eq:omegank:k0}. The stability of
such evaluations is analysed in detail in section 4. Some 
numerical experiments are presented in section 5, demonstrating the results
proven in this
work. In the appendix we collect those properties of Chebyshev polynomials used
in this paper. 

\section{Error estimates for the Product Clenshaw-Curtis rule}

The aim of this section is to derive convergence estimates for the error of the
quadrature rule \eqref{eq:the:rule:k}-\eqref{eq:the:rule:k:k0}. Obviously, 
\[
{\cal I}_{k,N}^\alpha- {\cal I}_k^\alpha(f)=\int_{-1}^1 E_N(x) \log \big(
(x-\alpha)^2\big)\,\exp(ikx)\,{\rm
d}x
\]
where
\begin{equation}
 \label{eq:def:E}
E_N:={\cal Q}_Nf-f, 
\end{equation} 
and  ${\cal Q}_N f$ is the interpolating polynomial at Chebyshev nodes cf.
\eqref{eq:QN}. 

A very popular technique when working with  Chebyshev polynomial approximations  is
to perform the change of 
variable $x=\cos\theta$. This transfers the problem  to the
frame of even  periodic functions and their approximations by trigonometric 
polynomials. Hence, if we denote  
$f_c(\theta):=f(\cos\theta)$ (note that $f_c$ is now even and $2\pi$-periodic)
we have that
\[
{\rm span}\: \langle\cos n\theta\ :\ n=0,\ldots N\rangle \ni\big({\cal Q}_N f 
\big)_c,\qquad 
\big({\cal Q}_N f \big)_c(n\pi/N)=f_c(n\pi/N),\quad n=0,\ldots,N. 
\]

Let us denote by $H^r(I)$  the classical  Sobolev space of order $r$ on
an interval $I\subset\R$ and define
\[
 H^r_\#:=\big\{\varphi\in H^r_{\rm loc}(\R)\: |\:
\varphi=\varphi(2\pi+\cdot)\big\}.
\]
{($H^r_{\rm loc}(\R)$ denotes here the space of functions which are locally in 
$H^r(\mathbb{R})$).}
The norm of these spaces can be characterised  in terms of the Fourier
coefficients of the elements as follows 
\begin{equation}
\label{eq:norm}
 \|\varphi\|^2_{H^r_\#}:=|\widehat{\varphi}(0)|^2+\sum_{n\ne 0}|n|^{2r}
|\widehat{\varphi}(n)|^2,\qquad
\widehat{\varphi}(n):=\frac{1}{2\pi}\int_{-\pi}
^\pi\varphi(\theta)\exp(-in\theta)\,
{\rm d}\theta. 
\end{equation}
If $r=0$, we just have the $L_2(-\pi,\pi)$ norm, whereas for { a positive integer $r$}
an equivalent norm is given by
\[
\bigg[ \int_{-\pi}^\pi |\varphi(\theta)|^2\, {\rm d }\theta+
\int_{-\pi}^\pi |\varphi^{(r)}(\theta)|^2\, {\rm d }\theta\bigg]^{1/2}.
\]

The convergence estimates for the trigonometric
interpolant in Sobolev
spaces (see for instance
\cite[\S 8.3]{SaVa:2002}) can be straightforwardly adapted to prove that
\begin{equation}
\label{eq:interpolatingEstimate}
\|(E_N\big)_c\|_{H^s_{\#}}\le C_{s,r_0}
N^{s-r}\|f_c\|_{H^r_\#}
\end{equation}
where  $r\ge s\ge 0$ with $r\ge r_0>1/2$ and $C_{s,r_0}$ 
independent of  $f$, $N$ and $r$. 

Set $w(x):=(1-x^2)^{1/2}$ and define for
$\beta\in\{-1,1\}$
\begin{eqnarray*}
 \|f\|_{\beta,w}&:=&\|f w^{\beta/2}\|_{L_2(-1,1)}=\bigg[\int_{-1}^1
|f(x)|^2  (1-x^2)^{\beta/2} {\rm d}x  \bigg]^{1/2}. 
\end{eqnarray*}
Notice in pass
\begin{equation}
 \label{eq:CSch}
\bigg|\int f(x)g(x)\,{\rm d}x\bigg|\le \|f\|_{-1,w}\|g\|_{1,w},\quad
\|g\|_{1,w}\le \sqrt{\pi}\|g\|_{L_\infty(-1,1)}. 
\end{equation}

From the relations  
\[
 \|f\|_{-1,w}=\|f_c\|_{L_2(0,\pi)},\qquad 
\|f'\|_{1,w}=\|(f_c)'\|_{L_2(0,\pi)},
\]
estimates \eqref{eq:interpolatingEstimate}, and the Sobolev embedding 
theorem \cite[Lemma 5.3.3]{SaVa:2002}, 
we can  easily derive the following estimate: For any $r\ge 1$, 
\begin{equation}
\label{eq:interp:estimates}
\|E_N\|_{-1,w}+ N^{-1}\|E_N'\|_{1,w}
+N^{ -1/2-\varepsilon}\|E_N\|_{L_\infty(-1,1)}
\le C_{\varepsilon} N^{-r}\|f_{c}\|_{H^{r}_\#},
\end{equation}
with $C_\varepsilon$ depending
only on $\varepsilon>0$. 

To prove the main result of this section  we previously need 
some technical results we collect in the next three Lemmas. The first
result concerns the asymptotics of 
\[
 \xi_0^\alpha(k)=\int_{-1}^1\exp(ikx)\log((x-\alpha)^2)\,{\rm d}x
\]
as $k\to\infty$.

\begin{lemma} \label{lemma:aux:conver2} For all $\alpha\in(-1,1)$ there exists
$C_\alpha>0$ so that for all $k\ge 2 $
 \[ 
|\xi_0^\alpha(k)|\le C_\alpha (1+{|\log(1-\alpha^2)|}) k^{-1}. 
 \]
Moreover, for $\alpha=\pm 1$, 
\[
 |\xi_0^{\pm 1}(k)|\le C_1 k^{-1}\log k,
\]
with $C_1>0$ independent of  $k\ge 2$. 
\end{lemma}
\begin{proof}  The result follows from working on the explicit expression for
$\xi_0^\alpha(k)$, see \eqref{eq:etaalpha_01}-\eqref{eq:eta_01}, and from using 
the limit of the functions involved as $k\to\infty$. We omit the proof for the
sake of brevity.

\end{proof}

The next Lemma complements the estimates given in \eqref{eq:interp:estimates}.

\begin{lemma}\label{lemma:aux:conver3}
 Let $E_N$ be given in \eqref{eq:def:E} and 
\begin{equation}\label{eq:00:lemma:aux:conver3}
e^\alpha_N(x):=\frac{E_N(x)-E_N(\alpha)}{x-\alpha}. 
\end{equation}
Then for all $r\ge s_0>5/2$, there exists $C_{s_0}$ independent of $f$, $N$ and
$r$ so that
\begin{equation}\label{eq:01:lemma:aux:conver3}
\|E_N'\|_{L_\infty(-1,1)}+ \|e^\alpha_N\|_{L_\infty(-1,1)}\le C_{s_0}N^{s_0-r}
\|f_c\|_{H^r_{\#}}.
\end{equation}
Moreover,
\begin{equation}\label{eq:02:lemma:aux:conver3}
\|wE_N''\|_{L_\infty(-1,1)} +
\|w(e^\alpha_{N})'\|_{L_\infty(-1,1)}\le C_{s_1}N^{s_1-r}
\|f_c\|_{H^r_{\#}},
\end{equation}
for $r\ge s_1>7/2$, with $C_{s_1}$ independent also of $f$, $N$ and $r$.
\end{lemma}
\begin{proof}  
Recall cf. \eqref{eq:Tnbounded}--\eqref{eq:Tnbounded:02}
 \begin{equation}\label{eq:04:lemma:aux:conver3}
  \| T_n'\|_{L_\infty(-1,1)}=n^2,\quad
\|wT_n'\|_{L_\infty(-1,1)}=n,\qquad 
\|wT_n''\|_{L_\infty(-1,1)}\le Cn^{3}
 \end{equation}
where $C>0$ is independent of $n$. 
Define now, 
\[
 p^\alpha_n(x):=\frac{T_{n}(x)-T_n(\alpha)}{x-\alpha}\in\mathbb{P}_{n-1}.
\]
Obviously 
 \begin{equation}\label{eq:05:lemma:aux:conver3}
  \|p_n^\alpha\|_{L_\infty(-1,1)}\le \|T_n'\|_{L_\infty(-1,1)}\le  n^2.
 \end{equation}
Besides, from \eqref{eq:Tn-Tn} (note that in the notation used there,
$U_j=T_{j+1}'/(j+1)$), we derive
\[
 (p_n^\alpha)'(x)=2\sum_{j=0}^{n-2} (j+1)^{-1} T'_{j+1}(\alpha)T_{n-1-j}'(x).
\] 
Then, using \eqref{eq:04:lemma:aux:conver3}
\begin{eqnarray}
 \|w (p_n^\alpha)'\|_{L_\infty(-1,1)}&\le& 2\sum_{j=0}^{n-2}
| (j+1)^{-1} T'_{j+1}(\alpha)|\:\|w
T_{n-1-j}'\|_{L_\infty(-1,1)}\nonumber\\
&=&
2\sum_{j=0}^{n-2}
  (j+1) (n-1-j)=\frac{1}{3} (n-1) n
(n+1)<\frac{n^3}{3}.\label{eq:06:lemma:aux:conver3}
\end{eqnarray}

On the other hand,  for all $f$ smooth enough, 
\begin{equation}
 \label{eq:ChebyshevSeries}
 f = \widehat{f_c} (0)+ 2\sum_{n=1}^\infty
\widehat{f_c}(n)T_n, \quad \widehat{f_c}(n)=\frac{1}{2\pi}\int_{-\pi}^\pi
\widehat{f_c}(\theta)\exp(i n\theta)\,{\rm
d}\theta=\frac{1}{\pi}\int_{-1}^1
 f  (x) T_n(x)\,{\rm d}x,
\end{equation}
{and, from \eqref{eq:norm}, 
\begin{equation}\label{eq:new}
 \|f_c\|^2_{H_{\#}^{r}}=|\widehat{f_c} (0)|^2+ 2\sum_{n=1}^\infty |\widehat{f_c} (n)|^2|n|^2
\end{equation}
}
To prove \eqref{eq:01:lemma:aux:conver3}, we  recall first the definition of 
$e_N^\alpha$ in \eqref{eq:00:lemma:aux:conver3}, and combine
\eqref{eq:ChebyshevSeries},  \eqref{eq:05:lemma:aux:conver3},
\eqref{eq:04:lemma:aux:conver3} {and \eqref{eq:new}}
to obtain
\begin{eqnarray*}
 \|e_N^\alpha\|_{L_\infty(-1,1)}+\|E_N' \|_{L_\infty(-1,1)}
\!\!&=&   {2}\Big\| \sum^\infty_{{n=1}}\!
\widehat{(E_N)_c}(n)
p_n^\alpha\Big\|_{L_\infty(-1,1)}+  {2}\Big\| \sum_{{n=1}}^\infty\!
\widehat{(E_N)_c}(n) T_n'\Big\|_{L_\infty(-1,1)}\\
\!\!&\le& {2}\sum_{n=1}^\infty\!
|\widehat{(E_N)_c}(n)|\big(
 \|p_n^\alpha\|_{L_\infty(-1,1)}+\|T_n'\|_{L_\infty(-1,1)}\big)\\
&\le&2\bigg[\sum_{n=1}^\infty  
n^{-1-2\epsilon }\bigg]^{1/2}
\bigg[{2}\sum_{n=1}^\infty
|\widehat{(E_N)_c}(n)|^2 n^{5+2\epsilon }\bigg]^{1/2}\\
&=&{2\bigg[\sum_{n=1}^\infty  
n^{-1-2\epsilon }\bigg]^{1/2}\big\|\big(E_N\big)_c\big\|_{H^{5/2+\epsilon}_\#}}
=:C_\varepsilon\big\|\big(E_N\big)_c\big\|_{H^{5/2+\epsilon}_\#}.
\end{eqnarray*}  
Estimate \eqref{eq:interpolatingEstimate} proves
\eqref{eq:01:lemma:aux:conver3}.
Proceeding
similarly, but using the last bound in \eqref{eq:04:lemma:aux:conver3} and
\eqref{eq:06:lemma:aux:conver3} instead, we prove
\eqref{eq:02:lemma:aux:conver3}.
\end{proof}

\begin{lemma}\label{lemma:aux:conver} There exists $C>0$ such that for any 
$\alpha\in[-1,1]$ and for all
$g\in H^1(-1,1)$ with $g(\alpha)=0$ 
\begin{eqnarray}
 \label{eq:01:lemma:aux:conver}
 |g(x)|&\le& C |x-\alpha|^{1/4}\|g'\|_{1,w},\qquad x\in[-1,1],\\
 \label{eq:02:lemma:aux:conver}
 \int_{-1}^1 \bigg|\frac{g(x)}{x-\alpha}\bigg|\,{\rm d }x&{\le}&
C\|g'\|_{1,w}.
\end{eqnarray}
\end{lemma}
\begin{proof}
Clearly, \eqref{eq:02:lemma:aux:conver} follows from
\eqref{eq:01:lemma:aux:conver}. 

Note first    
\[
 C:=\max_{\alpha\in[-1,1]} \Big\|
\frac{\arcsin()-\arcsin \alpha }{|\cdot-\alpha|^{1/2}}\Big\|_{L_\infty(-1,1)}
<\infty.
\]
Since $g(\alpha)=0$, it follows 
\begin{eqnarray*}
|g(x)|&=&\bigg|\int_\alpha^{x}  g'(s)\,{\rm d}s\bigg|\le 
\bigg[\int_{\alpha}^x \frac{{\rm d}
s}{\sqrt{1-s^2}}\bigg]^{1/2} \bigg[\int_{\alpha}^x |g'(s)|^2(1-s^2)^{1/2}\,{\rm
d}s\bigg]^{1/2}\nonumber\\
&\le& |\arcsin x -\arcsin\alpha|^{1/2} \bigg[\int_{-1}^1
|g'(s)|^2 w(s)\,{\rm
d}s\bigg]^{1/2}\le C|x-\alpha|^{1/4}\|g'\|_{1,w}.
\end{eqnarray*} 
 The result is then proven.
\end{proof}

We are ready to give the main result of this section which summarises the
convergence property of ${\cal I}_{k,N}(f)$ in terms of $N$, $k$ and the
regularity of $f$.

\begin{theorem} \label{theo:Conv} For all $\alpha\in[-1,1]$ there exists
$C_\alpha>0$ so that for $\delta\in\{0,1\}$
\begin{equation}\label{eq:Theo:conv:01}
|{\cal I}^\alpha_k(f)-{\cal I}^\alpha_{k,N}(f)|\le C_\alpha (1+k)^{-\delta}
N^{\delta-r}\|f_c\|_{H^r_\#}
\end{equation}
for all $r\ge 1$ and $k\ge 0$.  

Furthermore,  for all
$\varepsilon>0$ there exits $C_\varepsilon>0$ such that if $\alpha=\pm 1$ or 
$\alpha=0$ and $N$ is even it holds 
\begin{eqnarray}\label{eq:Theo:conv:03}
 |{\cal I}^{\alpha}_k(f)-{\cal I}^{\alpha}_{k,N}(f)|&\le& C_\varepsilon
(1+k)^{-2}(1+\alpha^2\log k) N^{7/2+\varepsilon-r}\|f_c\|_{ H^{r}_\#},
\end{eqnarray}
for all $f\in H_\#^r$ with $r>7/2+\varepsilon$
\end{theorem}
\begin{proof}Note first
\begin{eqnarray}
 |{\cal I}^\alpha_k(f)-{\cal I}^\alpha_{k,N}(f)|&=&\bigg|\int_{-1}^1
 E_N(x) \log \big( (x-\alpha)^2\big)\,\exp(ikx)\,{\rm
d}x\bigg|\le \|{\log\big(\cdot\,-\,\alpha\big)^2}\|_{1,w} 
\|E_N\|_{-1,w}\nonumber\\
&\le& C N^{-r}\|f_c\|_{H^r_\#}.
\label{eq:Theo:conv:03.5}  
\end{eqnarray}
where we have used   \eqref{eq:interp:estimates} (see also \eqref{eq:def:E}).
This  proves \eqref{eq:Theo:conv:01} for
$\delta=0$. Observe that from now on we can assume, 
without loss of generality, that $k\ge 1$. 

To obtain \eqref{eq:Theo:conv:01} for
$\delta=1$ we write
\begin{eqnarray}
|{\cal I}^\alpha_k(f)-{\cal I}^\alpha_{k,N}(f)|&=& \bigg|\int_{-1}^1
\big(E_N(x)-E_N(\alpha)\big)\log \big( (x-\alpha)^2\big)\,\exp(ikx)\,{\rm
d}x+ E_N(\alpha) \xi_0^\alpha(k)\bigg| \nonumber\\
&\le&\frac{1}{k}
\bigg[\bigg|(E_N(x)-E_N(\alpha))\log\big((x-\alpha)^2\big)\exp(ikx)\Big|_ { x=-1
}^{x=1}\bigg| \nonumber \\
&&+ \bigg|\int_{-1}^1
 E_N'(x)\log\big( (x-\alpha)^2\big) \exp(ikx)\,{\rm d}x\bigg| \nonumber \\
&&+ 2\bigg|\int_{-1}^1
\frac{E_N(x)-E_N(\alpha)}{x-\alpha} \exp(ikx)\,{\rm
d}x\bigg|+  {|E_N(\alpha)|}|k\xi_0^\alpha(k)|\bigg] \nonumber\\
&=:&\frac{1}{k}\big[R_N^{(1)}(k)+R_N^{(2)}(k)+R_N^{(3)}(k)+R_N^{(4)}(k)\big].
\label{eq:Theo:conv:04}
\end{eqnarray}
For bounding the first term, we make use of \eqref{eq:01:lemma:aux:conver} in 
Lemma \ref{lemma:aux:conver}  with
$g=E_N-E_N(\alpha)$:
\begin{equation}\label{eq:R1}
   R^{(1)}_{N}(k) \le
C\|E_N'\|_{1,w}|\big(|h_\alpha(-1)|+|h_\alpha(1)|\big),\quad
h_\alpha(x):=|x-\alpha|^{1/4}\log\big((x-\alpha)^2\big).
\end{equation}
where we have used that $h_\alpha\in{\cal
C}^0[-1,1]$
for any $\alpha\in[-1,1]$.

For the second term, notice that   \eqref{eq:CSch} and  
\eqref{eq:interp:estimates} imply 
\begin{equation}\label{eq:R2}
 R^{(2)}_{N}(k) \le  \|\log\big(\cdot\,-\,\alpha^2\big)\|_{-1,w} 
 \|E_N'\|_{1,w}.
\end{equation}

On the other hand,  \eqref{eq:02:lemma:aux:conver} of Lemma
\ref{lemma:aux:conver} yields
\begin{equation}\label{eq:R3}
 R^{(3)}_{N}(k) \le    C \|E_N'\|_{1,w}.
\end{equation}

Finally,   $E_N(\pm 1)=0$ and therefore $R^{(4)}_{N}(k)$ vanishes for
$\alpha=\pm 1$. Otherwise, 
Lemma \ref{lemma:aux:conver2} implies 
\begin{equation}\label{eq:R4}
  R^{(4)}_{N}(\alpha) \le  C'_\alpha |E_N(\alpha)|\le 
C'_\alpha\|E_N\|_{L_\infty(-1,1)}. 
\end{equation}
Bounding 
\eqref{eq:R1}--\eqref{eq:R4} with \eqref{eq:interp:estimates}, we derive
\eqref{eq:Theo:conv:01} for
$\delta=1$.

To prove
\eqref{eq:Theo:conv:03} we have to perform  another step of integration by
parts. First, we note, that, by hypothesis $E_N(\alpha)=0$ which implies
\begin{equation}\label{eq:R1R4}
 R_N^{(1)}(k)=
 R_N^{(4)}(k)=0. 
\end{equation}
 Thus, we just have to estimate $R_N^{(2)}(k)$ and
$R_N^{(3)}(k)$. For the first term, and proceeding as in
\eqref{eq:Theo:conv:04}, we obtain
\begin{eqnarray}
R_N^{(2)}(k)\!\!
&\le&\!\!
 \frac{1}{k}\bigg[\bigg|
(E'_N(x)-E'_N(\alpha))\log((x-\alpha)^2)\exp(ikx)\big|_{x=-1}^{x=1}
\bigg|\nonumber\\
&&+
\Big|\int_{-1}^1
E_N''(x)\log\big((x-\alpha)^2\big) \exp(ikx)\,{\rm d}x\Big|\nonumber\\
&&+2
\Big|\int_{-1}^1\frac{
E_N'(x)-E_N'(\alpha)}{ x-\alpha } \exp(ikx)\,{\rm d}x\Big|  +
\frac1k\big(|E_N'(\alpha)|\:|k\xi_0^\alpha(k)|\big) \nonumber\\
&=:&\frac{1}{k}\Big[S_N^{(1)}(k)+S_N^{(2)}(k)+S_N^{(3)}(k)+S_N^{(4)}(k)\Big].
\end{eqnarray}
Proceeding as in \eqref{eq:R2}-\eqref{eq:R3}, we  obtain
\[
  S_N^{(2)}(k) +
  S_N^{(3)}(k) \le C\|E_N''\|_{1,w}\le {C{\sqrt{\pi}}}\|w
E_N''\|_{L_\infty(-1,1)}\le 
{C_{\varepsilon}} 
N^{7/2+\varepsilon-r}\|f_c\|_{H_\#^r},
\]
where in the last inequality we have used 
\eqref{eq:02:lemma:aux:conver3} of Lemma
\ref{lemma:aux:conver3}.

On the other hand, 
$S_N^{(4)}(k)$ is bounded by applying again Lemma \ref{lemma:aux:conver2} (it is
just here where the $\log k$ term comes  up for $\alpha=\pm 1$). 

It only remains to study $S_N^{(1)}(k)$. Clearly,  for $\alpha=0$ it vanishes. 
For $\alpha= 1$, we have instead
\begin{eqnarray*}
 S_N^{(1)}(k)&\le& 4\log 2\|E'_N\|_{L_\infty(-1,1)}+\lim_{x\to 1}
|(E'_N(x)-E'_N(1))\log(x-1)^2|. 
\end{eqnarray*}
Since
\eqref{eq:01:lemma:aux:conver}
\[
|(E'_N(x)-E'_N(1))\log(x-1)^2|
\le   \|E_N''w\|_{L_\infty(-1,1)} \
\big|(1-x)^{1/4}\log\big((x-1)^2\big)\big|   \to
0,\quad \text{as }x\to 1,
\]
  \eqref{eq:01:lemma:aux:conver3} implies
\[
 S_N^{(1)}(k) \le 2\log 2\|E'_N\|_{L_\infty(-1,1)}\le C
N^{7/2-r}\|f_c\|_{H_{\#}^r}.
\]
The case $\alpha=-1$ is dealt with similarly.

In brief, we have proved that for $\alpha=\pm 1$, or for
$\alpha=0$ and $N$ is even, it holds
\begin{equation}
R_N^{(2)}(k)\le
{C_{\varepsilon}}{k}^{-1}
N^{7/2+\varepsilon-r}(1+\alpha^2\log
k)\|f_c\|_{H_\#^r}\label{eq:Theo:conv:05} 
\end{equation}
for  $\varepsilon>0$, $r\ge 7/2+\varepsilon$ with $C_\varepsilon$ independent of
$N$, $k$ and
$f$. 

Similarly, one can prove easily 
\begin{equation}
R_N^{(3)}(k) \le \frac{1}{k}\Big[ |e^\alpha_N(-1)|+|e^\alpha_N(1)| +
C \| (e_N^\alpha)'w\|_{L_\infty(-1,1)}\Big]\le
{C_{\varepsilon}}{k}^{-1}
N^{7/2+\varepsilon-r}\|f_c\|_{H_\#^r}\label{eq:Theo:conv:06}
\end{equation}
for all $r\ge7/2+\varepsilon$, and  $C_\varepsilon$ depending only on
$\varepsilon>0$. 

Plugging \eqref{eq:Theo:conv:05} and \eqref{eq:Theo:conv:06} into
\eqref{eq:Theo:conv:04} and recalling \eqref{eq:R1R4}, we have completed the
proof of \eqref{eq:Theo:conv:03}.
\end{proof}

 \begin{remark} 
We will show in the last section (see Experiment 5)  that the restriction of
$N$ to be even if $\alpha=0$ for achieving the $k^{-2}$-decay of the error  is
really needed. In the same experiment, we can check that the
error, specially for high values of $k$, is smaller for
$\alpha=0$ than for that obtained if $\alpha=1$. This supports empirically 
the fact that in the second case the $\log k$ term is certainly part of the
error term and therefore it affects, although very slightly, the convergence of 
the
rule. 
  \end{remark}

\section{Stable computation of the weights}

When the practical  implementation of the
quadrature rule is considered, we face that it essentially reduces to find a
way of evaluating 
$\xi^\alpha_n(k)$ cf. \eqref{eq:omegank}-\eqref{eq:omegank:k0} fast and
accurately.  In this section  we
present  the algorithms to carry out this evaluation and we
leave for the next one the proofs of the results concerning the stability of
such computations. 

For both, the oscillatory and non-oscillatory case,  what we will actually
compute is
\begin{equation}\label{eq:defetan}
 \eta_n^\alpha(k):=\int_{-1}^1 \log((x-\alpha)^2)U_n(x)\,\exp(ikx){\rm d}x
\end{equation}
where 
\begin{equation}
\label{eq:defUn}
U_n:=\frac{1}{n+1}T'_{n+1} 
\end{equation}
 is the Chebyshev polynomial of the second kind and degree $n$.  Notice that,
from this definition,  we have $U_{-1}=0$ and, according to that, we can
set
\[
  \eta_{-1}^\alpha(k):=0
\]
which simplifies some forthcoming expressions. From
\eqref{eq:UnTn:02} we have
\begin{equation}\label{eq:xiInTermsOfEta}
\xi_0^\alpha(k)= \eta_0^\alpha(k),\quad 
\xi_n^\alpha(k)=
\frac{1}{2}\big(\eta_n^\alpha(k)-\eta_{n-2}^\alpha(k)\big),\qquad n=1,2\ldots 
\end{equation} 
Observe that by \eqref{eq:Tnbounded}--\eqref{eq:UnIntegral}, there exists
$C>0$ such that for any $\alpha\in[-1,1]$ and $n$ 
 \begin{eqnarray}
 \label{eq:xi_is_bounded}
 |\xi^\alpha_n(k)|&\le& \int_{-1}^1
 \big|\log((x-\alpha)^2)\big|\,{\rm d}x
  \le C,\\
 |\eta^\alpha_n(k)|&\le& \bigg[\int_{-1}^1
 |U_n(x)|^2 {\sqrt{1-x^2}}\,   {\rm d}x
\,\bigg]^{1/2}  \bigg[\int_{-1}^1
 \big(\log((x-\alpha)^2)\big)^{2}\frac{{\rm d}x}
 {\sqrt{1-x^2}}\,\bigg]^{1/2}\le C.
 \label{eq:eta_is_bounded}
 \end{eqnarray}
That is,  these coefficients are bounded independent of $n$,
$\alpha$ and $k$.

\subsection{The non-oscillatory case}

Recall that for $k=0$, we have denoted $\xi_n^\alpha$ and
$\eta^\alpha_n $ instead of
$\xi^\alpha_n(0)$  and $\eta^\alpha_n(0)$ to lighten the
notation.

Assume that $\alpha\ne \pm 1$. Using the recurrence relation
for Chebyshev polynomials cf. \eqref{eq:TnRecc} and \eqref{eq:defUn}, we
deduce for $n\ge 1$  
\begin{eqnarray}
\eta_{n}^\alpha&=&\int_{-1}^1 U_n(x)\log\big((x-\alpha)^2\big)\,{\rm
d}x\nonumber\\
&=&
\int_{-1}^1 2x U_{n-1}(x)\log((x-\alpha)^2)\, {\rm
d}x-\int_{-1}^1
U_{n-2}(x)\log((x-\alpha)^2)\,{\rm d}x \nonumber\\
&=&\frac{2}{  n}\int_{-1}^1  (x-\alpha) T'_{ n}(x)\log((x-\alpha)^2)\, {\rm
d}x+2\alpha \eta_{n-1}^\alpha -\eta_{n-2}^\alpha\label{eq:01:etaalpha}. 
\end{eqnarray}
Integrating by parts in the first integral, we easily see that
\begin{eqnarray}
\int_{-1}^1  (x-\alpha) T'_{ n}(x)\log((x-\alpha)^2) {\rm
d}x
&=&\nonumber \\
&&\hspace{-5cm}=(x-\alpha)T_n(x)\log((x-\alpha)^2)\Big|_{x=-1}^{x=1}
-\int_{-1}^1 T_{n}(x)  \log((x-\alpha)^2) \,{\rm
d}x-2\int_{-1}^1 T_n(x)\,{\rm d}x\nonumber\\
&&\hspace{-5cm}= (1-\alpha)\log((1-\alpha)^2)+(-1)^{n }
(1+\alpha)\log((1+\alpha)^2)\nonumber\\
&&\hspace{-4.5cm}  -\frac12\big[
\eta_n^\alpha-\eta_{n-2}^\alpha\big]+
\left\{\begin{array}{ll}\frac{4}{n^2-1},\quad& \text{if $n$ is even},\\[1.2ex]
0,& \text{otherwise},
 \end{array}\right. \label{eq:01b:etaalpha}
\end{eqnarray}
where we have used that   $T_n(1)=1=(-1)^nT_{n}(-1)$, 
\eqref{eq:xiInTermsOfEta} (see also \eqref{eq:UnTn:02}) and  \eqref{eq:intTn}.
Hence, inserting \eqref{eq:01b:etaalpha}   in \eqref{eq:01:etaalpha} and 
resorting
appropriately the elements above, we arrive to the following {three}-terms
linear recurrence 
\begin{equation}
\label{eq:recPureLogcase}
 \eta_{n}^\alpha=\frac{2\alpha
n}{n+1}\eta_{n-1}^\alpha-\frac{n-1}{n+1}\eta_{n-2}^\alpha+\gamma_n^\alpha,
\qquad n=1,2,\ldots
\end{equation}
with
\begin{equation}\label{eq:gamma_n_alpha}
\gamma_n^\alpha:=\frac{4}{n+1}\left\{
\begin{array}{lcl}\displaystyle
(1-\alpha)\log(1-\alpha)+(1+\alpha)\log(1+\alpha)
+\frac{2}{n^2-1},\quad&\mbox{for
even
$n$},\\[1.25ex]
\displaystyle
(1-\alpha)\log(1-\alpha)-(1+\alpha)\log(1+\alpha),\quad &\mbox{for
odd $n$}.\\
\end{array}
\right.
\end{equation}
For $\alpha=\pm 1$, \eqref{eq:recPureLogcase} remains valid  with
 \begin{equation}
  \label{eq:gamma_n_1}
\gamma_n^{\pm 1}:=\frac{8}{n+1}\left\{
\begin{array}{lcl}\displaystyle
\log
2+\frac{1}{n^2-1},&\quad\mbox{for
even
$n$},\\[1.25ex]
\displaystyle
\mp \log 2, &\quad\mbox{for
odd $n$},
\end{array}
\right.
\end{equation}
which corresponds to take the limit as $\alpha\to \pm 1$ in 
\eqref{eq:gamma_n_alpha}.

Straightforward calculations show, in addition, that
\begin{equation}
\label{eq:etaalpha0} \eta_0^{\alpha}:=\left\{\begin{array}{ll}                 
       
-(\alpha -1) \log \left((\alpha
-1)^2\right)+(\alpha +1) \log \left((\alpha +1)^2\right)-4,
                         &\text{if }\alpha\ne \pm 1,\\
4\log 2-4,&\text{if }\alpha= \pm 1. 
\end{array}
\right.
\end{equation}

Recalling that $\eta_{-1}^\alpha=0$,  we are  ready to write down
the first algorithm.

\paragraph{Algorithm I: compute $\xi^\alpha_{n}$ for $n=0,1,\ldots,N $} 
\begin{enumerate}
\item Set $\eta_{-1}^\alpha=0$ and compute $\eta_0^\alpha$ according to
\eqref{eq:etaalpha0}.

\item For $n=1,\ldots,N$ 
\[
 \eta_{n}^\alpha=\frac{2\alpha
n}{n+1}\eta_{n-1}^\alpha-\frac{n-1}{n+1}\eta_{n-2}^\alpha+\gamma_n^\alpha,
\]
with $\gamma_n^\alpha$ defined in \eqref{eq:gamma_n_alpha}-\eqref{eq:gamma_n_1}.

\item Set
\[
 \xi_{0}^\alpha  =\eta_0^\alpha ,\qquad  
\xi_{n}^\alpha=
{\textstyle\frac{1}{2}}\big[\eta_{n}^\alpha-\eta_{n-2}^\alpha\big],\qquad
n=1,2,\ldots,N.
\]
\end{enumerate}

\begin{remark} For $\alpha=0$ the algorithm is even simpler since by parity
$\eta^0_{2n+1}=\xi^0_{2n+1}=0$  and  step 2 of the algorithm 
becomes 
\[
  \eta_{2n}^0= -\frac{2n-1}{2n+1}\eta_{2n-2}^0+\frac{8}{(2n+1)(4n^2-1)}.
\]

\end{remark}

\subsection{The oscillatory case}

Because of \eqref{eq:xiInTermsOfEta},    
\begin{eqnarray}
\frac{1}2\big(\eta_n^\alpha(k)- \eta_{n-2}^\alpha(k)\big)&=&
\xi^\alpha_n(k)=\int_{-1}^1
(T_n(x)-T_n(\alpha))\log((x-\alpha)^2)\exp(ikx)\,{\rm d}x
\nonumber \\ 
&&+T_n(\alpha)\int_{-1}^1 \log((x-\alpha)^2)\exp(ikx)\,{\rm
d}x.\label{eq:OscillatingCase01} 
\end{eqnarray}
Assume now that $\alpha\ne \pm 1$. Integrating by parts we
derive
\begin{eqnarray}
\int_{-1}^1
(T_n(x)-T_n(\alpha))\log((x-\alpha)^2)\exp(ikx)\,{\rm d}x
\nonumber\\
&&\hspace{-7cm}=\frac{1}{ik}\bigg[
(T_n(x)-T_n(\alpha))\log((x-\alpha)^2)\exp(ikx)\Big|_{x=-1}^{x=1}\nonumber\\
&&\hspace{-6.5cm} - \int_{-1}^1T_n'(x)
\log((x-\alpha)^2)\exp(ikx)\,{\rm d}x-
2\int_{-1}^1 \frac{T_n(x)-T_n(\alpha)}{x-\alpha}\exp(ikx)\,{\rm
d}x\bigg]\label{eq:OscillatingCase02-0} 
\\
&&\hspace{-7cm}=\frac{1}{ik}\bigg[
(1-T_n(\alpha))\log((1-\alpha)^2)\exp(ik
)+((-1)^{n+1}+T_n(\alpha))\log((1+\alpha)^2)\exp(-ik 
)\nonumber 
\\
&&\hspace{-6.5cm} -n \int_{-1}^1 U_{n-1}(x) \log((x-\alpha)^2)\exp(ikx)\,{\rm
d}x\nonumber\\
&&\hspace{-6.5cm} -
 2\int_{-1}^1 U_{n-1}(x)\exp(ikx)\,{\rm d}x-4\sum_{j=0}^{n-2}
T_{n-1-j}(\alpha)\int_{-1}^1
U_{j}(x)\exp(ikx)\,{\rm d}x\bigg].
\label{eq:OscillatingCase02} 
\end{eqnarray}
(We have applied \eqref{eq:Tn-Tn} to write the last integral 
in \eqref{eq:OscillatingCase02-0} as the sum in the right-hand-side of 
\eqref{eq:OscillatingCase02}).

Inserting \eqref{eq:OscillatingCase02} in \eqref{eq:OscillatingCase01} and using
\eqref{eq:defUn} we derive the following recurrence equation
\begin{equation}
 \label{eq:recurrence}
\eta_{n}^\alpha(k)-\frac{2n}{ik}\eta_{n-1}^\alpha(k)+\eta^\alpha_{n-2}(k)=
\gamma_{n}^\alpha(k)
\end{equation}
where  
\begin{eqnarray}
\gamma_n^\alpha(k)&:=&\frac{2}{ik}
\big[(1-T_n(\alpha))\log((1-\alpha)^2)\exp(ik)+((-1)^{n+1}+T_n(\alpha))
\log((1+\alpha)^2
)\exp(-ik)\big]\nonumber\\
&&-\frac{4}{ik}\bigg[2\sum_{j=0}^{n-2}T_{n-1-j}
(\alpha)\rho_j(k)+{\rho_{n-1}(k)}\bigg]+{2T_n(\alpha)\eta_{0}^\alpha(k)}, 
\label{eq:gammanalphak}
\end{eqnarray}
with
\[
 \rho_j(k):=\int_{-1}^1 U_j(x)\,\exp(ikx)\,{\rm d}x, \quad j=0,\ldots, n-1.
\]
Let us point out that $(\rho_j(k))_{j=0}^N$ can be computed in ${\cal O}(N)$
operations (see
\cite{DoGrSm:2010}). 

 For $\alpha=\pm
1$ we obtain the same recurrence \eqref{eq:recurrence} with
\begin{eqnarray}
 \gamma_n^{\pm 1}(k)\!\!&=&\!\!\frac{4}{ik}\Big[
\left\{\begin{array}{ll}
       \log(4)\exp(\mp ik),\quad& \text{if $n$ is odd}\nonumber\\
0,\quad&\text{otherwise}
       \end{array}
\right. \\
&&-2\sum_{j=0}^{n-2}(\pm 1)^{n-j+1}
\int_{-1}^1 U_j(x)\exp(ikx) {\rm d}x-\int_{-1}^1 U_{n-1}(x)\exp(ikx) {\rm d
}x\Big].\label{eq:gammanalphak01}
\end{eqnarray}

It just remains to compute
$\eta_{0}^\alpha(k)$ for setting up the algorithm. For
this purpose  we introduce the sine and cosine integral functions
\[
 \Si(t):=\int_0^t\frac{\sin x}{x}\,{\rm d}x,\qquad
 \Ci(t):=\gamma+ \log(t)+ \int_0^t\frac{\cos x-1}{x}\,{\rm d}x, 
\]
with $\gamma\approx 0.57721$ the Euler-Mascheroni constant. Straightforward
calculations  show that
\begin{eqnarray}
\eta_0^\alpha(k)=\xi_0^\alpha(k)&=&\frac{2}{k}\Big[\log(1-\alpha^2)\sin k
+\sin(\alpha k)\big(\Ci((\alpha+1)k))-\Ci((1-\alpha)k)\big)\nonumber\\
&&
-\cos(\alpha k)\big(\Si((\alpha+1)k)+\Si((1-\alpha)k)\big)
\Big]\nonumber\\
&&+\frac{2i}{k}\Big[ \log\big(\frac{1+\alpha}{1-\alpha}\big)\cos k 
+\cos(\alpha k) \big(\Ci((1-\alpha)k)-\Ci((1+\alpha)k)\big)\nonumber\\
&&
-\sin(\alpha
k)\big(\Si((1-\alpha)k)+\Si((1+\alpha)k)\Big],\label{eq:etaalpha_01}
\end{eqnarray}
for $\alpha\ne \pm 1$, and
\begin{eqnarray}
\eta_{0}^{\pm 1}(k)=\xi_0^{\pm 1}(k)&=&\frac{2}k \Big[-(\gamma-\Ci(2k)
+\log(k/2))\sin k-\Si(2k)\cos k\Big]\nonumber\\
&&\pm \frac{2i}k\Big[(\gamma-\Ci(2k) +\log(2k))\cos k-
                \Si(2k)\sin k\Big].\label{eq:eta_01}
\end{eqnarray}
 
From now on, we will denote by $\lfloor x \rfloor$ the floor function, i.e, the
largest integer smaller than $x$.

\paragraph{Algorithm II: computation of $\xi_n^\alpha (k)$ for $n=0,\ldots,N$
with
$N\le 
\lfloor k \rfloor-1$} 
\begin{enumerate}
\item Set $\eta_{-1}^\alpha(k)=0$ and evaluate   $\eta_{0}^\alpha(k)$ according
to
\eqref{eq:etaalpha_01}--\eqref{eq:eta_01}.
\item Compute $\gamma^\alpha_n(k)$ for $n=1,\ldots, N$ using
\eqref{eq:gammanalphak}-\eqref{eq:gammanalphak01}.

\item For $n=1,2,\ldots,N$, define
\begin{equation}
 \label{eq:02:alg2}
 \eta^\alpha_{n}(k)=\gamma_n^\alpha(k)-\frac{2n}{ik}\eta^\alpha_{n-1}
(k)+\eta^\alpha_{ n-2}
(k).
\end{equation}

\item Set
\[
 \xi^\alpha_{0}(k) = \eta^\alpha_0(k),\qquad 
\xi^\alpha_{n}(k) = \frac{1}{2}\big[\eta^\alpha_n(k)-\eta^\alpha_{n-2}(k)\big],
\quad n=1,\ldots,N.
\]
\end{enumerate}

Observe that we have restricted the range for which this algorithm can be
 used to $N\le k-1$. This is because  the 
recurrence relation \eqref{eq:02:alg2}, as it will be shown in the next 
sections,
is not longer stable for $n>k$. Thus, we have to explore different ways to
compute $\eta^\alpha_n(k)$ when $n\ge k$.

Then assume that $N>\lfloor k\rfloor-1$. 
Note that Algorithm II returns
$\eta_{0}^\alpha(k),\ldots, \eta_{\lfloor k\rfloor-1}^\alpha$. In order to 
compute the
remaining weights we still use \eqref{eq:02:alg2} but rewriting it in a
different way, namely
as a tridiagonal system. (This is the so-called
Oliver method cf. \cite{Oliver1966}). Hence, let
\begin{equation}
\label{eq:matrix_rhs}
A^\alpha_N(k):=
 \begin{bmatrix}
  \frac{2(\lfloor k\rfloor+1)}{ik} & 1\\
-1 &\frac{2(\lfloor k\rfloor+2)}{ik} &1\\
&\ddots&\ddots&\\
&&-1&\frac{2N-2}{ik}
  \end{bmatrix}
\quad {\bf b}^\alpha_{N}(k)=
 \begin{bmatrix}
 \eta^\alpha_{\lfloor k\rfloor-1}+\gamma^\alpha_{\lfloor k\rfloor+1}(k)\\
\gamma^\alpha_{\lfloor k\rfloor+2}(k)\\
\vdots\\
\gamma^\alpha_{N-1}(k) -\eta_{N}^\alpha(k)
 \end{bmatrix}.
\end{equation}
Note that $A^\alpha_N(k)$ is row dominant, which  implies first that the system 
is
uniquely solvable and next suggests that all the calculations
become stable. This will be rigorously proven in next section. 

In short, if $\bm{\eta}$ solves 
\begin{equation}
\label{eq:system}
A^\alpha_N(k)\bm{\eta}= {\bf b}^\alpha_{N}(k),
\end{equation}
necessarily 
\[
 \bm{\eta}=
 \begin{bmatrix}
 \eta^\alpha_{\lfloor k\rfloor}(k) &
\eta^\alpha_{\lfloor k\rfloor+1}(k)& \cdots&
\eta^\alpha_{N-1}(k)
 \end{bmatrix}^\top.
\]
In the definition of the right-hand-side we
find $\eta^\alpha_N(k)$ and 
$\eta^\alpha_{\lfloor k\rfloor-1}$. The latter
is already known.  
Thus, only
the problem of finding $\eta^\alpha_N(k)$ remains open. For these purposes,  as
in \cite{BrHa:2007}, we will use the
Jacobi-Anger expansion
cf.
\cite[\S 2.2]{Watson1995}, \cite[(9.1.44-45)]{AbrSt}:
\[
  \exp(ikx)= J_0(k) + 2\sum_{n=1}^\infty i^n J_{n}(k) T_{n}(x),
\]
where $J_n$ is the Bessel function  of the first kind and order $n$. 
Hence, using \eqref{eq:ProdTnUm}, we derive
\begin{eqnarray*}
\eta^\alpha_{N}(k)&=&\int_{-1}^1 U_{N}(x)\exp(i k x)\log((x-\alpha)^2)\,{\rm
d}x\nonumber\\
&=&J_0(k)\int_{-1}^1
U_N(x)\log((x-\alpha)^2)\,{\rm d}x +2\sum_{m=1}^\infty i^n J_{m}(k)\int_{-1}^1
U_N(x){T_m}(x)\log((x-\alpha)^2)\,{\rm d}x \nonumber\\
&=& J_0(k)\eta^\alpha_N+ 
\sum_{m=1}^N i^m    
J_{m}(k)\big(\eta^\alpha_{N+m}+\eta^\alpha_{N-m} \big) +
\sum_{m=N+1}^\infty i^m     
J_{m}(k)\big(\eta^\alpha_{N+m}-\eta^\alpha_{m-N-2} \big).
\end{eqnarray*}
Observe that the coefficients $\eta_n^\alpha$ can be obtained from
Algorithm I. 
By \eqref{eq:eta_is_bounded}, in order to estimate how many terms are needed to
evaluate this coefficient, we need to estimate how fast $J_M(k)$ decays as
$M\to\infty$. We point out 
cf. \cite[(9.1.10), (9.3.1)]{AbrSt}, 
\begin{equation}
 \label{eq:AsymptBessel}
 J_{M}(k) \approx \frac{1}{M!}\Big(\frac{k}2\Big)^{M}\approx \frac{1}{\sqrt{2\pi
M}}\Big(\frac{ek}{2 M}\Big)^{M} 
\end{equation}
which shows that $ J_{M}(k)$ decreases very fast as $n\to\infty$. In addition, 
it
suggests that taking
$  \approx k $ terms in the series
(\ref{eq:AsymptBessel}), should be enough to
approximate $\eta_{N}(k)$ within the machine precision.

In
our implementation we have taken 
\begin{equation}
 \eta^\alpha_{N}(k)\approx   J_0(k)\eta^\alpha_N(k)+ 
\sum_{n=1}^N i^n     
J_{n}(k)\big(\eta^\alpha_{N+n}-\eta^\alpha_{N-n} \big) +
\sum_{n=N+1}^{M(k)} i^n     
J_{n}(k)\big(\eta^\alpha_{N+n}-\eta^\alpha_{n-N-2}
\big)\label{eq:asymptotic}.
\end{equation}
with $M(k)=25+\lceil e k/2\rceil$ which has demonstrated  to
be
sufficient for our purposes.

\paragraph{Algorithm III: compute $\xi^\alpha_n(k)$ for $n=\lfloor
k\rfloor,\ldots,
N$} 
\begin{enumerate}

\item Construct ${\bf b}^\alpha_N(k)$ using
\begin{enumerate}
\item $\eta^\alpha_{\lfloor k\rfloor-1}(k)$ returned in  Algorithm II

\item $\gamma^\alpha_n(k)$ for $n=\lfloor k\rfloor+1,\ldots, N$ defined in
\eqref{eq:gammanalphak}-\eqref{eq:gammanalphak01}.

\item   $\eta^\alpha_{N}(k)$  evaluated with the sum 
\eqref{eq:asymptotic}. 
\end{enumerate}
\item Construct the tridiagonal matrix $A^\alpha_N(k)$ defined in
\eqref{eq:matrix_rhs} and solve 
\[
A^\alpha_N(k)\bm{\eta}= {\bf b}^\alpha_{N}(k).
\]
Set  
\[
\eta^\alpha_{\lfloor k \rfloor-1+\ell}(k)=(\bm\eta)_{\ell},\qquad
\ell=1,\ldots,
N- \lfloor k \rfloor.
\]
\item Set
\begin{eqnarray*}
\xi_{n}^\alpha(k)&=&\frac{1}{2}\big[\eta^\alpha_n(k)-\eta_{n-2}^\alpha(k)\big],
\qquad
n=\lfloor k\rfloor,\ldots, N.
\end{eqnarray*}
\end{enumerate}

\subsubsection*{On the computational cost}
Certainly, one could use \eqref{eq:asymptotic} for
computing {\em all} the coefficients $(\eta_n^{\alpha}(k))_n$, as it was 
proposed in {\rm \cite{BrHa:2007}} (for $\alpha=0$). However, this choice
results in
a more expensive algorithm. By restricting this approach to the last 
coefficient, and
only if $N>k$, we
can speed up the algorithm since all  the  terms  but the last one, are computed
by solving a tridiagonal system which can be done in  
${\cal O}(N-\lfloor k \rfloor)$ operations by Thomas algorithm. 

The vector  $\left(\gamma_n^\alpha(k)\right)_{n=1}^N$ can be also 
constructed very fast. Hence, note that  the bulk part in 
\eqref{eq:gamma_n_alpha} 
is the convolution of the
vectors
\[
 \begin{pmatrix} T_n(\alpha)\end{pmatrix}_{n=0}^{N-1},\quad 
 \begin{pmatrix} \rho_n \end{pmatrix}_{n=0}^{N-1}
\]
which can be done in ${\cal O}(N\log(N))$ operations by using FFT. (For
$\alpha\in\{-1,0,1\}$ this could be achieved even faster  from
\eqref{eq:gamma_n_alpha}, since $T_n(\pm 1)=(\pm 1)^n$ and $T_n(0)=1$  if $n$
is even and $0$ otherwise).

Another possible bottleneck of the algorithm could be found in the evaluation
of the Bessel functions $J_n(k)$.
Let us show how it can be overcome. We recall that the
Bessel functions obey the recurrence relation 
\begin{equation}
 \label{eq:BesselRecurrence}
 J_{n+1}(k)-\frac{2n}{k}J_n(k)+J_{n-1}(k)=0.
\end{equation}
Notice in pass that it is very similar to that
obtained in \eqref{eq:recurrence} for evaluating our coefficients. Thus,
we can exploit these similarities to get a faster evaluation of these
functions: Once $J_0(k)$ and
$J_1(k)$ are  evaluated by usual methods, 
\eqref{eq:BesselRecurrence}  can be safely used for
evaluating $J_n(k)$ for $n\le \lfloor k\rfloor$. For the remainder values,
i.e. for $n\ge \lfloor
k\rfloor+1$, we  make use of the Oliver approach  and solve 
\[
 \begin{bmatrix}
  -\frac{2(\lfloor k\rfloor+1)}{k} & 1\\
1 &-\frac{2(\lfloor k\rfloor+2)}{k} &1\\
&\ddots&\ddots&\\
&&1&-\frac{2M(k)}{k}
  \end{bmatrix}
 \begin{bmatrix}
J_{\lfloor k\rfloor+1}(k)\\
J_{\lfloor k\rfloor+2}(k)\\
\vdots\\
J_{ M(k)}(k)
 \end{bmatrix}=
 \begin{bmatrix}
 - J_{\lfloor k\rfloor-1}(k)\\
0\\
\vdots\\
-J_{ M(k)+1}(k)
 \end{bmatrix}.
\] 
The asymptotics \eqref{eq:AsymptBessel} can be used to  approximate
$J_{M(k)+1}(k)$,   which gives even better results that setting simply  
$J_{M(k)+1}(k)\approx 0$.
The evaluation turns out to be stable just for the  same reasons  that ensure
the stability of Algorithms II and III (see next section).
 

\section{Numerical stability}
\label{sec:4}

We analyse the stability of the algorithms separately in three
propositions and collect the stability results for  Algorithms II and III, when
they
work together,  in a theorem which ends this section. 

The (usually small) parameter $\varepsilon_j>0$ will be used in this section to
represent any possible  perturbation occurring in the evaluation such as
round-off errors or errors coming from previous computations.

\begin{theorem}\label{theo:stability01}
 Let $\bm{\varepsilon}_N:=\big(\varepsilon_0,\varepsilon_1,\ldots,
\varepsilon_{N}\big)\in\mathbb{\R}^{N+1}$ with $\|\bm{\varepsilon}\|_\infty\le
\varepsilon$ and define the sequence
\begin{eqnarray*}
\eta^{\alpha,\varepsilon}_{-1}&=&\eta^{\alpha }_{-1}=0,\qquad
\eta^{\alpha,\varepsilon}_{0} = \eta^{\alpha}_{0}+\varepsilon_0, \\
\eta^{\alpha,\varepsilon}_{n}
&=& \gamma_n^ \alpha+\frac{2\alpha n}{n+1}\eta_{ n-1 }^{\alpha,\varepsilon}
-\frac{ n-1}{ n+1}\eta_{
n-2}^{\alpha,\varepsilon}
+\varepsilon_n, \quad \text{for }
n=1,2,\ldots. 
\end{eqnarray*}
Then for all $N>0$
\[
|\eta^{\alpha}_{N}-\eta_{N}^{\alpha,\varepsilon}|\le
\bigg[\frac{1}{N+1}\sum_{j=0}^{N} (j+1)|U_{N-j}(\alpha)| \bigg]\varepsilon\le
{\textstyle\frac{1}{6}} (N+2) (N+3)\varepsilon.
\]
\end{theorem}
\begin{proof}
Clearly, 
\[
 \eta_{N}^\alpha-\eta_{N}^{\alpha,\varepsilon}=
\sum_{j=0}^{N} \delta_N^{(j)}, 
\]
where $\delta_{n}^{(j)}$ is  given by 
\begin{eqnarray}
\delta_{j-1}^{(j)} &:=&0 ,\quad
\delta_{j}^{(j)}  := \varepsilon_j, \qquad
\delta_{n}^{(j)}: = \frac{2\alpha
n}{n+1}\delta_{n-1}^{(j)}-\frac{n-1}{n+1}\delta_{n-2}^{(j)},\quad
n=j+1,j+2,\ldots
\label{eq:02:prop:stability01}
\end{eqnarray}
It is easy to check, using \eqref{eq:TnRecc}, that the solution of the problem
above is given by
\[
 \delta_n^{(j)}=\frac{j+1}{ n+1} U_{n-j}(\alpha) \varepsilon_j.
\]
Therefore, using \eqref{eq:Tnbounded}
\begin{eqnarray*}
 | \eta_{N}^\alpha-\eta_{N}^{\alpha,\varepsilon}|&\le& \frac{1}{N+1}
\bigg[\sum_{j=0}^N (j+1)|U_{N-j}(\alpha)|\bigg]\varepsilon\le \frac{1}{N+1}
\bigg[\sum_{j=0}^N (j+1) (N-j+1)\bigg]\varepsilon\\
&=&
{\textstyle\frac{1}{6}} (N+2) (N+3)\varepsilon.
\end{eqnarray*}
The proof is now finished. 
\end{proof}

\begin{remark} \label{remark:stability01} In view of this result, we  conclude
that theoretically 
$\alpha=0$ turns out to be  the most stable case. Indeed,  since
$U_{2j}(0)=(-1)^j$, 
\[
  | \eta_{2N}^0-\eta_{2N}^{0,\varepsilon}|\le  \frac{\varepsilon}{2N+1}
\sum_{j=0}^N (2j+1)=\frac{(N+1)^2\varepsilon }{2N+1}.
\]
(Note that $\eta_{2j+1}^0=0$ and therefore only $(\eta_{2j}^0)_j$ has to be
considered). 

On the other hand, $\alpha=\pm 1$ are precisely the most unstable cases, since 
$|U_j(\pm 1)|=j+1$. We point out, however, that in practical computation the
algorithm has demonstrated, see 
section 5, that: (a) the computation  is stable for
$\alpha\in[-1,1]$, much better than that theory predicts;
(b) the error observed for $\alpha=0$  is a little
smaller than
that for $\alpha=\pm 1$.
\end{remark}

Next we consider the stability of Algorithms II and III, i.e., the
oscillatory case. 

\begin{proposition}\label{prop:Stability02}
 Let $N\le k-1$ and set 
$\bm{\varepsilon}=(\varepsilon_0,\ldots,\varepsilon_N)\in\mathbb{C}^{N+1}$
with $\|\bm{\varepsilon}\|_{\infty}\le
\varepsilon$ and consider the sequence
\begin{eqnarray*}
\eta^{\alpha,\varepsilon}_{-1}(k)&=&\eta^{\alpha}_{-1}(k)=0,\\
\eta^{\alpha,\varepsilon}_0(k)&=&\eta_0^{\alpha}(k)+\varepsilon_0,\\
\eta^{\alpha, \varepsilon}_{n}(k)&=&\gamma_n^{\alpha}(k)-\frac{2n}{ik}
\eta^{\alpha, \varepsilon}_{n-1} (k)+\eta^{\alpha, \varepsilon}_{n-2}
(k)+\varepsilon_{n},\quad n=1,2,\ldots,N. 
\end{eqnarray*}
Then, for all $0\le n< k-1$.
\begin{equation}
\label{eq:01:prop:Stability02}
|\eta^{\alpha,\varepsilon}_{n}(k)-
{\eta}^\alpha _{n}(k)|\le 
\bigg[1+\frac{4}{3}\frac{(n+1)k^{1/2}}{(k^2-(n+1)^2)^ { 1/4 } }
\bigg]\varepsilon.
\end{equation}
Therefore, for all $n\le k-2$.
\begin{equation}
\label{eq:02:prop:Stability02}
 |\eta^{\alpha,\varepsilon}_{n}(k)-
{\eta}^\alpha _{n}(k)|\le \Big[1+\frac{2^{3/4}}{3} (n+1)^{3/4}k^{1/2}
\Big]\varepsilon,
\end{equation}
whereas for $k-2<n\le k-1$, i.e., for $n=\lfloor k\rfloor-1$, 
\begin{equation}
\label{eq:03:prop:Stability02}
  |\eta^{\alpha,\varepsilon}_{\lfloor k\rfloor-1}(k)-
\eta^\alpha_{\lfloor k\rfloor-1}(k)| \le \big[4+ 2^{7/4}
k^{5/4}\big]\varepsilon.
\end{equation}
\end{proposition}
\begin{proof}
As before, it suffices to study the sequence
\begin{eqnarray*}
 \delta_{-1}&=&0, \qquad \delta_0=\varepsilon_0\\
 \delta_n&=&-\frac{2n}{ik}
  \delta _{n-1}+  \delta _{n-2}+\varepsilon_{n},\quad n=1,2,\ldots,N.
\end{eqnarray*}
We refer now to \cite[Th. 5.1]{DoGrSm:2010} where the stability of this
sequence is analysed and whose  proof can be straightforwardly adapted to
derive 
\eqref{eq:01:prop:Stability02} 

To prove \eqref{eq:02:prop:Stability02}, we observe  that
\eqref{eq:01:prop:Stability02} implies that for $n\le k-2$,
\begin{eqnarray*}
 |\eta^{\alpha,\varepsilon}_{n}(k)-
 \eta^\alpha_{n}(k)|&\le& 
\bigg[1+\frac{4}{3}\frac{(n+1)k^{1/2}}{((n+2)^2-(n+1)^2)^ { 1/4 }
}\bigg]\varepsilon< 
\bigg[1+\frac{4}{3}\frac{(n+1)k^{1/2}}{( 2n+3)^{ 1/4 }
}\bigg]\varepsilon\\
&\le& 
\big[1+\frac{2^{7/4}}{3} (n+1)^{3/4}k^{1/2} \big]\varepsilon.
\end{eqnarray*}

If $ n=\lfloor k\rfloor-1$, we can use  \eqref{eq:02:prop:Stability02} as
follows
\begin{eqnarray*}
  |\eta^{\alpha,\varepsilon}_{n}(k)-
 \eta^\alpha_{n}(k)|&\le& \varepsilon+\frac{2n}{k} 
|\eta^{\alpha,\varepsilon}_{n-1}(k)-
{\eta}^\alpha_{n-1}(k)| +  |\eta^{\alpha,\varepsilon}_{n-2}(k)-
{\eta}^\alpha_{n-2}(k)|\\
&\le& \big[4+ 2\cdot \frac{2^{7/4}}3 n^{3/4} k^{1/2} + \frac{2^{7/4}}3
(n-1)^{3/4} 
k^{1/2} \big]\varepsilon\le \big[4+ 2^{7/4} n^{3/4} k^{1/2}\big]\varepsilon\\
&\le & \big[4+  2^{7/4} k^{5/4} \big]\varepsilon. 
\end{eqnarray*}
Bound \eqref{eq:03:prop:Stability02} is now proven.
\end{proof}

The stability of Algorithm III is consequence of the next result.

\begin{proposition}\label{prop:Stability03}
Let  $N>k$ and consider the solutions of the original and perturbed systems
\[
A^\alpha_N(k)\bm{\eta}={\bf b}^\alpha_N(k),\qquad  (A_{N}^\alpha(k)+\Delta
A^\alpha_N(k))\bm{\eta}^\varepsilon={\bf b}^\alpha_N(k) +\Delta {\bf
b}^\alpha_N(k).
\]
Then, if $(k+2)\|\Delta A^\alpha_N(k)\|_\infty<2$, it holds
\[
\|\bm{\eta}^\varepsilon-\bm{\eta}\|_{\infty}\le
\frac{k+2 }{2-(k+2)\|\Delta A^\alpha_N(k)\|_\infty}\big[
\|\Delta{\bf b}^\alpha_N(k)\|_\infty
 +\|\Delta A^\alpha_N(k)\|_\infty\|\bm{\eta}\|_{\infty}
\big].
\]
\end{proposition}
\begin{proof} A classical result in stability theory for  
systems of linear equations (see
for instance \cite[Th. 8.4]{MR1007135}) states that 
\begin{equation}
\label{eq:03:prop:Stability03}
 \|\bm{\eta}^\varepsilon-\bm{\eta}\|_{\infty}\le
\frac{\big\|\big(A^\alpha_N(k)\big)^{-1}
\big\|_{\infty}}{1-\|\Delta A^\alpha_N(k)\|_\infty
\big\|\big(A^\alpha_N(k)\big)^{-1}\big\|_{\infty}}
\big[\|\Delta{\bf
b}^\alpha_N(k)\|_{\infty}+\|\Delta
A^\alpha_N(k)\|_\infty\|\bm{\eta}\|_\infty\big]. 
\end{equation}
Thus, we just have to estimate   $\|\big(A^\alpha_N(k)\big)^{-1}\|_\infty$. 
Let 
\begin{eqnarray*}
 D_N(k)&=&\frac{2}{ik}\begin{bmatrix}
                \lfloor k\rfloor+1&\\
&\lfloor k\rfloor+2&\\
&&\ddots&\\
&&&N-1 \end{bmatrix}, \\ 
K_N(k)&=&
\begin{bmatrix}
0&\frac{i k}{2(\lfloor k\rfloor+2)}&\\
-\frac{i k}{2(\lfloor k\rfloor+1)}&0 &\frac{i k}{2(\lfloor k\rfloor+3)}\\
&-\frac{i k}{2(\lfloor k\rfloor+2)}&0&\frac{i k}{2(\lfloor k\rfloor+4)}\\
&&\ddots&\ddots&\ddots\\
&&&\ddots&\ddots&\ddots\\
&&&&-\frac{ik}{2(N-2)}&0
\end{bmatrix}.
\end{eqnarray*}
Let $I$ denote the identity matrix. Clearly, it holds
\[
 A^\alpha_N(k)=(I +K_N(k))D_N(k)\quad \Rightarrow \quad
\big(A^\alpha_N(k)\big)^{-1} =\big(D_N(k)\big)^{-1}(I +K_N(k))^{-1}.
\]
Notice also \begin{eqnarray*}
 \|(D_N(k))^{-1}\|_{\infty}&=&\frac{k}{2(\lfloor k\rfloor+1)}<\frac{1}{2}\\
\|K_N(k)\|_{\infty}&=&\frac{k}{2(\lfloor
k\rfloor+1)}+\frac{k}{2(\lfloor k\rfloor+3)}<\frac12
+\frac{k}{2(k+2)}= \frac{k+1}{k+2}.
\end{eqnarray*}
Collecting these inequalities, we conclude
\begin{equation}
 \label{eq:norm_invA}
 \big\|\big(A^{\alpha}_N(k)\big)^{-1}\big\|\le\frac{
\|(D_N(k))^{-1}\|_{\infty}}{1-\|K_N(k)\|_{\infty}}<                             
\frac{ k+2 } { 2 }.   
\end{equation}
Inserting \eqref{eq:norm_invA}  in
\eqref{eq:03:prop:Stability03} the result is proven.
\end{proof}

The perturbation $\Delta A^\alpha_N(k)$ in the matrix is essentially  
round-off errors. Since $A^\alpha_N(k)$ is a tridiagonal matrix we can safely 
expect 
$(k+2)\|\Delta A^\alpha_N(k)\|_\infty<< 1$.

The last result of this section  states the numerical stability of the
algorithm in the oscillatory case and is result of  combining appropriately
propositions \ref{prop:Stability02} and
\ref{prop:Stability03}. 

\begin{theorem}\label{theo:Stab}
 With the notations of  Propositions
\ref{prop:Stability02}--\ref{prop:Stability03}, it holds
\begin{itemize}
 \item[{\rm (a)}] For any $\nu\in(0,1)$  there exists $C_\nu$, depending only on
$\nu$, so that for $N<\nu k$, 
\begin{equation}
 \label{eq:01:theo:Stab}
 \max_{n=0,\ldots,N }|\eta_n^{\alpha,\varepsilon}(k)-\eta^\alpha_n(k)|\le
C_\nu N 
\varepsilon. 
\end{equation}
with $C_\nu$ independent of $k$ and $N$. 

 \item[{\rm (b)}]  There exists $C>0$ independent of $k$ and $\alpha$ so that 
\begin{equation}
 \label{eq:02:theo:Stab}
 \max_{n=0,\ldots,\lfloor k\rfloor -1
}|\eta^{\alpha,\varepsilon}_n(k)-\eta^\alpha_n(k)|\le
C k^{5/4}\varepsilon.
\end{equation} 
 \item[{\rm (c)}] For $N>k$, if the following conditions are, in addition,
satisfied 
\[
 \|\Delta
A^\alpha_N(k)\|_\infty\le \varepsilon, \qquad \|\Delta {\bf
b}^\alpha_N(k)\|_\infty \le
\varepsilon+|
\eta_{\lfloor k\rfloor-1}^{\alpha}(k)-\eta^{\alpha,\varepsilon}_{\lfloor
k\rfloor-1}(k)|,
\]
and  $\varepsilon<(k+1)^{-1}$, then   it holds 
\begin{eqnarray}
 \label{eq:03:theo:Stab}
 \max_{j=\lfloor k\rfloor-1,\ldots,N
}|\eta^{\alpha,\varepsilon}_j(k)-\eta^\alpha_j(k)|&\le&  (k+2)\big(1+
|
\eta_{\lfloor
k\rfloor-1}^{\alpha}(k)-\eta^{\alpha,\varepsilon}_{\lfloor
k\rfloor-1}(k)|+\|\bm{\eta}\|_\infty \big)\varepsilon\qquad \\
&\le& C' k^{9/4}\varepsilon\label{eq:04:theo:Stab}
\end{eqnarray}
where $C$ is independent of $k$ and $N$.
\end{itemize}
\end{theorem}
\begin{proof} Estimates \eqref{eq:01:theo:Stab}-\eqref{eq:02:theo:Stab}
follow from Proposition \ref{prop:Stability02}.  For item  (c) 
 \eqref{eq:03:theo:Stab} is a
consequence of  Proposition \ref{prop:Stability03}. Finally,
\eqref{eq:04:theo:Stab} is obtained by applying \eqref{eq:02:theo:Stab}, which
bounds the last term in
 \eqref{eq:03:theo:Stab},  and using that  $\|\bm{\eta}\|_\infty$ is
uniformly bounded independent of $N$ and $\alpha$ cf. \eqref{eq:eta_is_bounded}.
\end{proof}

Let us emphasise that, in our computations, 
\eqref{eq:04:theo:Stab} has been demonstrated to be very pessimistic. 

\begin{table}
 \[ 
\begin{array}{l|c |c}
&  e^0_{\rm abs}(n)& e^0_{\rm rel}(n)\\
\hline&&\\
n=10&     1.11{\rm E}{-16} &  4.33{\rm E}{-16}\\
n=200&     1.11{\rm E}{-16} &  1.07{\rm E}{-15}\\
n=400&    1.11{\rm E}{-16} &  1.31{\rm E}{-15}
  \end{array}               
 \]
\caption{\label{tab:01} Absolute and relative errors in $\xi^0_n$ for different
values of
$n$}

 \[ 
\begin{array}{l|c |c}
&  e^1_{\rm abs}(n)& e^1_{\rm rel}(n)\\
\hline&&\\
n=10&     5.83{\rm E}{-16} &  3.70{\rm E}{-14}\\
n=200&    5.83{\rm E}{-16} &  7.92{\rm E}{-14}\\
n=400&    5.83{\rm E}{-16} &  3.90{\rm E}{-13}
  \end{array}               
 \]
\caption{\label{tab:01b} Absolute and relative errors $\xi^1_n$ for different
values of
$n$}
\end{table}

\section{Numerical Experiments}

We collect in this section some numerical experiments to illustrate the
theoretical results presented in this paper. The implementation of the rule, for
$\alpha=0,-1$ is available in \cite{web}.

\label{sec:5}
\subsection{Stability for $\xi_n^\alpha$}

We have computed here $\xi_{n}^\alpha$ for $n=0,\ldots, 100$ using the
implementation of our method in Matlab. Next, we compare the numerical results 
with that
obtained using symbolic calculations in Mathematica, which will be denoted
by  $\xi_{n}^{\alpha, \rm Symb}$. The evaluation of these expressions is done  
using (very) high arithmetic precision to keep the round-off errors well below
the significant digits returned in our implementation in Matlab. 

We present
\[
 e^\alpha_{\rm abs}(N)=\max_{n=0,\ldots, N} |\xi^\alpha_{{n}}-\xi_{{n}}^{\alpha, \rm
Symb}|,\qquad
 e^\alpha_{\rm rel}(N)=\max_{n=0,\ldots, N}
\frac{|\xi^\alpha_{{n}}-\xi_{{n}}^{\alpha,\rm Symb}|}{|\xi_{{n}}^{\alpha,\rm Symb}| }
\]
in Tables  \ref{tab:01} (for $\alpha=0$) and  \ref{tab:01b} (for $\alpha=1$)
for different values of $n$.  We clearly see that for all $n$ the  error is very
close to the machine's unit round off and that the
results returned for $\alpha=0$ are slightly better than that for $\alpha=1$.
This should indicate that
Theorem \ref{theo:stability01} is sharp (see also Remark
\ref{remark:stability01}).

\subsection{Stability for $\xi^\alpha_n(k)$}

As before, we compare here the values of $
 \eta^\alpha_n(k)$ computed by our code with that returned by   Mathematica. The
results are shown for $\alpha=0$ in  Table \ref{tab:03}
and for $\alpha=1$ in Table \ref{tab:04}.

\begin{table}
 \[ 
\begin{array}{r|c |c|c |c|c }
    N\setminus k   & 10&20&40&80&160\\
\hline&&&&&\\
  1 & 1.39{\rm E}{-17}  &  1.04{\rm E}{-17} &  1.30{\rm E}{-18}  & 
4.34{\rm E}{-19} & 2.71{\rm E}{-020} \\
 10 & 1.33{\rm E}{-15}  &  2.22{\rm E}{-16} &  2.78{\rm E}{-17}  & 
2.78{\rm E}{-17}  & 0.00   \\
 20 & 1.67{\rm E}{-16}  &  6.66{\rm E}{-16} &  0.00 & 2.78{\rm E}{-17}  &  6.94{\rm E}{-18}  \\  
 40 & 2.78{\rm E}{-17}  &  1.39{\rm E}{-16} &  1.11{\rm E}{-15}  & 
4.16{\rm E}{-17}  & 0.00   \\
 80 & 2.78{\rm E}{-17}  &  1.39{\rm E}{-17} &  5.55{\rm E}{-17}  & 
1.11{\rm E}{-15}  &  2.08{\rm E}{-17}    \\
160 & 0.00 &  1.39{\rm E}{-17} &  2.78{\rm E}{-17}  & 7.63{\rm E}{-17}  &  1.55{\rm E}{-15}   
  \end{array}               
 \]
 \[ 
\begin{array}{r|c |c|c |c|c }
   N\setminus k & 10&20&40&80&160\\
\hline&&&&&\\
 1 & 1.58{\rm E}{-16}  &  1.64{\rm E}{-15} &  6.18{\rm E}{-16}  & 
2.70{\rm E}{-16}  &  1.28{\rm E}{-16}    \\
10 & 6.48{\rm E}{-16}  &  4.63{\rm E}{-16} &  1.63{\rm E}{-16}  & 
3.42{\rm E}{-16}  &  0.00   \\
20 & 3.91{\rm E}{-16}  &  3.61{\rm E}{-16} &  0.00 &  3.14{\rm E}{-16}  &  1.76{\rm E}{-16}    \\
40 & 1.65{\rm E}{-16}  &  6.60{\rm E}{-16} &  6.76{\rm E}{-16}  & 
3.56{\rm E}{-16}  &  0.00   \\
80 & 3.48{\rm E}{-16}  &  1.66{\rm E}{-16} &  5.30{\rm E}{-16}  & 
7.66{\rm E}{-16}  &  8.18{\rm E}{-16}\\   
{-16} & 0.00 &  3.48{\rm E}{-16} &  6.63{\rm E}{-16}  &  1.46{\rm E}
{-15} &  1.23{\rm E}{-15}  
  \end{array}               
 \]
 \caption{\label{tab:03} Absolute (top)   and  relative (below) error in
computing $\xi^0_n(k)$}
\end{table}

\begin{table}
 \[ 
\begin{array}{r|c |c|c |c|c }
 N\setminus k   & 10&20&40&80&160\\
\hline&&&&&\\
  1 & 3.86{\rm E}{-16}  &  1.39{\rm E}{-17} &  2.95{\rm E}{-16}  & 
2.78{\rm E}{-17}  & 1.39{\rm E}{-17}  \\  
 10 & 2.54{\rm E}{-15}  &  2.24{\rm E}{-16} &  4.79{\rm E}{-16}  & 
2.86{\rm E}{-17}  &  1.55{\rm E}{-17}  \\  
 20 & 2.22{\rm E}{-17}  &  1.56{\rm E}{-15} &  1.25{\rm E}{-15}  & 
2.08{\rm E}{-17}  & 4.39{\rm E}{-17}  \\  
 40 & 1.03{\rm E}{-16}  &  3.71{\rm E}{-17} &  4.10{\rm E}{-15}  & 
1.67{\rm E}{-16}  &  1.12{\rm E}{-16}  \\  
 80 & 1.81{\rm E}{-17}  &  6.35{\rm E}{-17} &  2.70{\rm E}{-17}  & 
1.60{\rm E}{-15}  &  1.31{\rm E}{-16}  \\  
160 & 1.32{\rm E}{-17}  &  1.86{\rm E}{-17} &  1.00{\rm E}{-16}  & 
5.02{\rm E}{-16}  & 8.68{\rm E}{-16}     
  \end{array}               
 \]
 \[ 
\begin{array}{r|c |c|c |c|c }
  N\setminus k  & 10&20&40&80&160\\
\hline&&&&&\\
  1 & 5.47{\rm E}{-16}  &  3.05{\rm E}{-17} &  1.35{\rm E}{-15}  & 
1.89{\rm E}{-16}  &  2.15{\rm E}{-16}  \\  
 10 & 4.55{\rm E}{-15}  &  5.19{\rm E}{-16} &  2.00{\rm E}{-15}  & 
1.87{\rm E}{-16}  & 1.95{\rm E}{-16}  \\  
 20 & 7.89{\rm E}{-16}  &  4.42{\rm E}{-15} &  3.57{\rm E}{-15}  & 
1.04{\rm E}{-16}  & 4.97{\rm E}{-16}  \\  
 40 & 1.18{\rm E}{-14}  &  3.84{\rm E}{-15} &  1.89{\rm E}{-14}  & 
4.12{\rm E}{-16}  & 5.95{\rm E}{-16}  \\  
 80 & 6.94{\rm E}{-15}  &  2.23{\rm E}{-14} &  9.93{\rm E}{-15}  & 
1.21{\rm E}{-14}  & 6.56{\rm E}{-16}  \\  
160 & 1.73{\rm E}{-14}  &  2.28{\rm E}{-14} &  1.27{\rm E}{-13}  & 
6.02{\rm E}{-13}  & 1.01{\rm E}{-14}   
  \end{array}               
 \]
 \caption{\label{tab:04} Absolute (top)   and  relative (below) error in
computing $\xi^1_n(k)$.}
\end{table}

It is worth  mentioning that in our 
implementation in Matlab we face  an annoying bug. Algorithm II (and
therefore indirectly Algorithm  III)
makes use of the sine and cosine integral  functions ($\Si$ and
$\Ci$ in our notation) just for evaluating the first
coefficient $\eta_0^\alpha(k)$. 
These functions are only included in Matlab as part of  the symbolic
toolbox,  and therefore it is not presented in all distributions. Moreover, any
call to these functions consumes a significant CPU time because of the own 
nature
of the symbolic toolbox. Hence, in
some of our experiments we observed that when using the built-in functions
almost half of the CPU time was consumed in performing these two evaluations.  

Thus we wrote our own implementation for sine and cosine
integral functions. The evaluation is accomplished by a combination of
asymptotic expansion for large arguments \cite[(5.2.34)-(5.2.35)]{AbrSt} and a
sum of Bessel functions of
fractional
order for small and moderate arguments \cite[(5.2.15)]{AbrSt}. Despite our
efforts, our code introduces a very small error in the last or in the
last but one significant digits. However, such errors only  affect 
the first
few  coefficients very slightly and do not propagates to the rest of 
coefficients.
Hence, it gives us a unwanted proof of the stability
of the algorithm.

\begin{table}[p]
\[
\begin{array}{r|c |c|c |c|c |c |c }
   N\setminus k     &   0 &   10 &   10^2 & 10^3 &10^4&10^5 \\ 
\hline&&&&&&\\       
 11 & 1.71{\rm E}{-03} &  4.00{\rm E}{-03} &  1.75{\rm E}{-04} & 
1.82{\rm E}{-05} & 1.83{\rm E}{-06} &  1.83{\rm E}{-07} \\
 12 & 4.56{\rm E}{-05} &  3.28{\rm E}{-04} &  1.44{\rm E}{-06} & 
1.37{\rm E}{-08} & 1.37{\rm E}{-10} & 1.37{\rm E}{-12} \\
 23 & 1.65{\rm E}{-08} &  2.56{\rm E}{-08} &  4.80{\rm E}{-09} & 
3.89{\rm E}{-10} & 3.80{\rm E}{-11} &  3.80{\rm E}{-12} \\
 24 & 2.96{\rm E}{-10} &  8.24{\rm E}{-09} &  9.93{\rm E}{-10} & 
9.09{\rm E}{-12} &  9.09{\rm E}{-14} & 9.08{\rm E}{-16} \\
 47 & 6.66{\rm E}{-16} &  1.11{\rm E}{-16} &  8.97{\rm E}{-17} & 
1.29{\rm E}{-17} & 1.08{\rm E}{-19} &  1.36{\rm E}{-020} \\
 48 & 6.66{\rm E}{-16} &  2.73{\rm E}{-16} &  8.85{\rm E}{-17} & 
1.26{\rm E}{-17} & 1.08{\rm E}{-19} &  2.71{\rm E}{-020} 
\end{array}
\] 
\[
\begin{array}{r|c |c|c |c|c |c |c }
    N\setminus k   &   0 &   10 &   10^2 & 10^3 &10^4&10^5 \\ 
\hline&&&&&&\\
 11 & 9.38{\rm E}{-04} &  5.49{\rm E}{-03} &  2.78{\rm E}{-03} & 
2.89{\rm E}{-03} &  2.91{\rm E}{-03} &  2.91{\rm E}{-03} \\
 12 & 2.50{\rm E}{-05} &  4.50{\rm E}{-04} &  2.28{\rm E}{-05} & 
2.18{\rm E}{-06} &  2.17{\rm E}{-07} & 2.18{\rm E}{-08} \\
 23 & 9.04{\rm E}{-09} &  3.51{\rm E}{-08} &  7.61{\rm E}{-08} & 
6.20{\rm E}{-08} &  6.05{\rm E}{-08} &  6.04{\rm E}{-08} \\
 24 & 1.63{\rm E}{-10} &  1.13{\rm E}{-08} &  1.58{\rm E}{-08} & 
1.45{\rm E}{-09} & 1.45{\rm E}{-10} &  1.45{\rm E}{-11} \\
 47 & 3.66{\rm E}{-16} &  1.52{\rm E}{-16} &  1.42{\rm E}{-15} & 
2.05{\rm E}{-15} & 1.73{\rm E}{-16} &  2.17{\rm E}{-16} \\
 48 & 3.66{\rm E}{-16} &  3.75{\rm E}{-16} &  1.40{\rm E}{-15} & 
2.01{\rm E}{-15} & 1.73{\rm E}{-16} &  4.32{\rm E}{-16} 
\end{array} 
\]
\caption{Absolute (top) and relative (bottom) errors for integral \eqref{eq:exp5} with
$\alpha=0$\label{table:intalpha0}}

\[
\begin{array}{l|c |c|c |c|c |c |c }
 N\setminus k      & 0 & 10 &  10^2 & 10^3 & 10^4& 10^5 \\ 
\hline&&&&&&\\
11& 1.81{\rm E}{-05} &  8.89{\rm E}{-04} &  3.04{\rm E}{-05} & 
5.04{\rm E}{-07} &  6.33{\rm E}{-09} & 7.90{\rm E}{-11} \\
12 & 2.43{\rm E}{-06} &  7.72{\rm E}{-05} &  8.94{\rm E}{-06} & 
1.74{\rm E}{-07} &  1.77{\rm E}{-09} & 2.15{\rm E}{-11} \\
23  & 4.21{\rm E}{-11} &  2.60{\rm E}{-11} &  1.50{\rm E}{-09} & 
5.51{\rm E}{-12} & 1.25{\rm E}{-13} & 1.48{\rm E}{-15} \\
24 & 5.25{\rm E}{-11} &  4.91{\rm E}{-11} &  1.89{\rm E}{-09} & 
1.84{\rm E}{-11} &  2.81{\rm E}{-13} & 3.55{\rm E}{-15} \\
47 & 1.04{\rm E}{-18} &  7.85{\rm E}{-17} &  9.22{\rm E}{-17} & 
2.47{\rm E}{-17} &  2.09{\rm E}{-18} &  1.10{\rm E}{-19} \\
48 & 7.31{\rm E}{-17} &  8.89{\rm E}{-17} &  9.17{\rm E}{-17} & 
2.17{\rm E}{-17} &  1.89{\rm E}{-18} &  1.12{\rm E}{-19} 
\end{array} 
\]
\[
\begin{array}{l|c |c|c |c|c |c |c }
 N\setminus k      &   0 &   10 &   10^2 & 10^3 &10^4&10^5 \\ 
\hline&&&&&&\\
11 & 8.09{\rm E}{-04} &  3.07{\rm E}{-03} &  1.00{\rm E}{-03} & 
1.71{\rm E}{-04} & 1.26{\rm E}{-05} & 1.27{\rm E}{-06} \\
12 & 1.09{\rm E}{-04} &  2.67{\rm E}{-04} &  2.94{\rm E}{-04} & 
5.92{\rm E}{-05} &  3.53{\rm E}{-06} & 3.45{\rm E}{-07} \\ 
23 &
1.89{\rm E}{-09} &  9.00{\rm E}{-11} & 4.92{\rm E}{-08} & 1.87
{\rm E}{-09} & 2.49{\rm E}{-10} &  2.39{\rm E}{-11} \\ 
24 & 2.30{\rm E}{-09} & 1.70{\rm E}{-10} & 6.22{\rm E}{-08} & 6.26{\rm E}{-09} &
5.61{\rm E}{-10} & 5.72{\rm E}{-11} \\  47 & 4.67{\rm E}{-17} &
2.71{\rm E}{-16} & 3.03{\rm E}{-15} & 8.38{\rm E}{-15} & 4.17
{\rm E}{-15} &  1.77{\rm E}{-15} \\ 
 48 & 3.27{\rm E}{-15} &  3.07{\rm E}{-16} &  3.02{\rm E}{-15} & 
7.37{\rm E}{-15} &  3.78{\rm E}{-15} &  1.80{\rm E}{-15} 
\end{array} 
\]
\caption{Absolute (top) and relative (bottom) errors for integral \eqref{eq:exp5} with
$\alpha=1$\label{table:intalpha1}}
\end{table}

\subsection{Experiments for  an oscillatory integral }

Let 
\begin{equation}\label{eq:exp5}
I_{\alpha}(k):=\int_{-1}^1  \frac{\cos(4
x)}{x^2+x+1}\log\big((x-\alpha)^2\big)\, \exp(ik x)\,{\rm d}x
\end{equation}
We have computed the errors returned by our numerical method for different
values of $k$, $N$ and for $\alpha=0$ (Table \ref{table:intalpha0}) and
$\alpha=1$ (Table
\ref{table:intalpha1}). As exact value we just have used that returned by the rule
when  a huge number of points is used. 

Several facts can be observed right from the beginning. First, the convergence
is very fast: with modest values of $N$  we get approximations with an error of
the same order  as the round-off unity. Second, if we read  Table  \ref{table:intalpha0} 
by rows, we
clearly see that  {for $\alpha=0$, and even $N$, the error decreases as $k^{-2}$. For odd $N$, however,  the errors of the rule decay only as $k^{-1}$.
(And therefore, the relative error keeps bounded independent of $k$ in this last case)}. 

Such phenomenon doest not occur when $\alpha=1$, i.e, when the logarithmic
singularity occurs at the end of the interval (See Table
\ref{table:intalpha1}): {The error for fixed even or odd $N$ decays as 
$k^{-2}$}.

\begin{table}[p]\small
\[
\begin{array}{l|c |c|c |c|c |c |c }
 N\setminus k    &   0 &1&   10 &   10^2 & 10^3 &10^4&10^5 \\ 
\hline&&&&&&&\\
11  &   2.59{\rm E}-03    &    2.59{\rm E}-03   &   6.90{\rm E}-03    &    2.80{\rm E}-04&   
2.81{\rm E}-05    &    2.83{\rm E}-06&    
2.83{\rm E}-07 \\ 
   12  &   1.72{\rm E}-04    &    1.72{\rm E}-04   &   4.82{\rm E}-03    &    3.38{\rm E}-05&   
3.23{\rm E}-07    &    2.97{\rm E}-09&   
 2.92{\rm E}-11 \\ 
   23  &   2.86{\rm E}-04    &    2.86{\rm E}-04   &   2.96{\rm E}-04    &    4.80{\rm E}-05&   
6.48{\rm E}-06    &    6.56{\rm E}-07&   
 6.57{\rm E}-08 \\ 
   24  &   1.07{\rm E}-05    &    1.07{\rm E}-05   &   1.29{\rm E}-04    &    2.24{\rm E}-05&   
1.90{\rm E}-07    &    1.53{\rm E}-09&   
 1.50{\rm E}-11 \\ 
   47  &   3.36{\rm E}-05    &    3.36{\rm E}-05   &   3.18{\rm E}-05    &    3.04{\rm E}-05&   
1.50{\rm E}-06    &    1.57{\rm E}-07&   
 1.58{\rm E}-08 \\ 
   48  &   6.64{\rm E}-07    &    6.64{\rm E}-07   &   6.91{\rm E}-06    &    1.20{\rm E}-05&   
1.05{\rm E}-07    &    7.61{\rm E}-10&   
 7.78{\rm E}-12 \\ 
   95  &   4.07{\rm E}-06    &    4.07{\rm E}-06   &   3.91{\rm E}-06    &    5.62{\rm E}-05&   
4.14{\rm E}-07    &    3.84{\rm E}-08&   
 3.86{\rm E}-09 \\ 
   96  &   4.13{\rm E}-08    &    4.13{\rm E}-08   &   4.17{\rm E}-07    &    9.36{\rm E}-05&   
5.52{\rm E}-08    &    3.03{\rm E}-10&   
 4.05{\rm E}-12
\end{array}
\]
\[
\begin{array}{l|c |c|c |c|c |c |c }
 N\setminus k    &   0 &1&   10 &   10^2 & 10^3 &10^4&10^5 \\ 
\hline&&&&&&&\\
   11  &   5.07{\rm E}-05    &    5.05{\rm E}-05   &   1.34{\rm E}-04    &    5.33{\rm E}-06&   
5.51{\rm E}-07    &    5.56{\rm E}-08& 
 5.56{\rm E}-09 \\ 
   12  &   2.66{\rm E}-06    &    2.65{\rm E}-06   &   8.00{\rm E}-05    &    4.74{\rm E}-07&   
4.72{\rm E}-09    &    4.78{\rm E}-11& 
 4.76{\rm E}-13 \\ 
   23  &   1.32{\rm E}-06    &    1.32{\rm E}-06   &   1.37{\rm E}-06    &    2.34{\rm E}-07&   
2.99{\rm E}-08    &    3.03{\rm E}-09& 
 3.03{\rm E}-10 \\ 
   24  &   4.41{\rm E}-08    &    4.41{\rm E}-08   &   5.47{\rm E}-07    &    9.01{\rm E}-08&   
6.59{\rm E}-10    &    6.27{\rm E}-12&   
 6.17{\rm E}-14 \\ 
   47  &   3.74{\rm E}-08    &    3.74{\rm E}-08   &   3.55{\rm E}-08    &    3.35{\rm E}-08&   
1.69{\rm E}-09    &    1.76{\rm E}-10& 
 1.76{\rm E}-11 \\ 
   48  &   7.01{\rm E}-10    &    7.01{\rm E}-10   &   7.38{\rm E}-09    &    1.25{\rm E}-08&   
1.01{\rm E}-10    &    8.11{\rm E}-13&  
 7.83{\rm E}-15 \\ 
   95  &   1.11{\rm E}-09    &    1.11{\rm E}-09   &   1.07{\rm E}-09    &    1.54{\rm E}-08&   
1.12{\rm E}-10    &    1.05{\rm E}-11& 
 1.06{\rm E}-12 \\ 
   96  &   1.10{\rm E}-11    &    1.10{\rm E}-11   &   1.12{\rm E}-10    &    2.50{\rm E}-08&   
1.36{\rm E}-11    &    9.53{\rm E}-14& 
 9.99{\rm E}-16 \\ 
\end{array}                                        
\]
\caption{\label{Tab:06:01}Errors of the quadrature rule for integral \eqref{eq:exp:06a} with 
$\alpha=0$, $\beta=1/2$ (top) and $\beta=3/2$ (bottom)}
\end{table}

\subsection{Non-smooth functions}

In this last experiment we run our code to compute
\begin{subequations}
\label{eq:exp:06}
\begin{eqnarray}
 I_1(k,\alpha,\beta)&:=&\int_{-1}^1(1+x)^\beta\log\big((x-\alpha)^2\big)\exp(i k x)\,{\rm d}x, 
\label{eq:exp:06a}\\
 I_0(k,\alpha,\beta)&:=&\int_{-1}^1|1/2+x|^\beta\log\big((x-\alpha)^2\big)\exp(i k x)\,{\rm d}x, 
\label{eq:exp:06b}
\end{eqnarray}
\end{subequations}
for $\alpha\in\{-1,0\}$ and $\beta\in\{1/2,3/2\}$ to analyse how precise are the regularity 
assumptions in the hypothesis Theorem \ref{theo:Conv}. 

We  expect the convergence of
the rule  to be faster for the first integral since, after 
performing the cosine change of variables, $|1+\cos \theta |^\beta\in 
H^{2\beta+1/2-\varepsilon}_\#$ whereas $|1/2+\cos \theta |^\beta\in 
H^{\beta+1/2-\varepsilon}_\#$. The regularity of the 
transformed function is precisely what  appears in the estimate of Theorem \ref{theo:Conv} 
(function $f_c$ in the right-hand-side). 

We show in Tables \ref{Tab:06:01}-\ref{Tab:06:04} the error of the rule for different values of 
$k$ and $N$. (The 
{\em exact} integral was computed by using the Clenshaw-Curtis rule on graded meshes towards the 
singular points,  cf. \cite{KiDoGrSm:09}). Clearly, the errors are in almost all cases smaller for 
\eqref{eq:exp:06a} than for \eqref{eq:exp:06b}.

It is difficult to estimate the order of convergence of the rule because it becomes chaotic as $k$ 
increases in such a way that the larger is $k$, the bigger has to be $N$ to make the error 
decay steady to zero.   
{Hence, the results
in  the first columns of Table \ref{Tab:06:01}-\ref{Tab:06:02} point out to a convergence of
order $4$ and $6$, approximately, for $\beta=1/2$ and $\beta=3/2$, 
much higher than that the theory predicts, $1.5$ and $3.5$ respectively. On the other hand, 
the results in Tables \ref{Tab:06:03}-\ref{Tab:06:04}, corresponding to the integral \eqref{eq:exp:06b}, 
suggest that the rules converges with order $1.5$ and $2.5$, which should be compare with that derived
from our results, $1$ and $2$. This could indicate that the convergence results proven in this paper
can be somewhat pessimistic.}


If we read the table by rows, we can detect that the 
${\cal O}(k^{-2})$ decay of the error occurs only in Table 
\ref{Tab:06:01} and in \ref{Tab:06:03} for $\beta=3/2$. Only for the first integral 
\eqref{eq:exp:06a} with $\beta=3/2$, this property has been rigorously proved 
since in the notation of  Theorem \ref{theo:Conv}  $f_c\in H^{7/2-\varepsilon}_\#$. There  is 
however no theoretically justification for the other cases and it certainly deserves more attention 
to study if the regularity assumptions 
can be relaxed for $\alpha=0$.

On the other hand, the error  does not behave as ${\cal O}(k^{-2})$ in Table 
\ref{Tab:06:02} although for $\beta=3/2$ Theorem \ref{theo:Conv} should imply such decay of 
the error. We think that the very irregular convergence of the rule in this case could force $N$ 
and $k$  to be larger to observe it.

\begin{table}[p]\small
\[
\begin{array}{l|c |c|c |c|c |c |c }
 N\setminus k    &   0 &1&   10 & 10^2 & 10^3 &10^4&10^5 \\ 
\hline&&&&&&&\\
   11  &   4.02{\rm E}-03    &    4.02{\rm E}-03   &   1.86{\rm E}-03    &    4.99{\rm E}-03&   
3.34{\rm E}-04    &    1.56{\rm E}-05&    6.38{\rm E}-07 \\ 
   12  &   3.15{\rm E}-03    &    3.15{\rm E}-03   &   3.54{\rm E}-03    &    4.73{\rm E}-03&   
3.29{\rm E}-04    &    1.56{\rm E}-05&    6.37{\rm E}-07 \\ 
   23  &   5.11{\rm E}-04    &    5.11{\rm E}-04   &   5.11{\rm E}-04    &    2.17{\rm E}-03&   
2.79{\rm E}-04    &    1.48{\rm E}-05&    6.27{\rm E}-07 \\ 
   24  &   4.54{\rm E}-04    &    4.54{\rm E}-04   &   4.53{\rm E}-04    &    1.96{\rm E}-03&   
2.74{\rm E}-04    &    1.47{\rm E}-05&    6.26{\rm E}-07 \\ 
   47  &   6.83{\rm E}-05    &    6.83{\rm E}-05   &   6.83{\rm E}-05    &    2.20{\rm E}-04&   
1.88{\rm E}-04    &    1.32{\rm E}-05&    6.04{\rm E}-07 \\ 
   48  &   6.43{\rm E}-05    &    6.43{\rm E}-05   &   6.43{\rm E}-05    &    1.90{\rm E}-04&   
1.85{\rm E}-04    &    1.31{\rm E}-05&    6.03{\rm E}-07 \\ 
   95  &   9.28{\rm E}-06    &    9.28{\rm E}-06   &   9.28{\rm E}-06    &    1.02{\rm E}-05&   
4.78{\rm E}-05    &    1.04{\rm E}-05&    5.61{\rm E}-07 \\ 
   96  &   9.00{\rm E}-06    &    9.00{\rm E}-06   &   9.00{\rm E}-06    &    1.38{\rm E}-05&   
4.59{\rm E}-05    &    1.04{\rm E}-05&    5.61{\rm E}-07
 \end{array}                                         
\]
\[
\begin{array}{l|c |c|c |c|c |c |c }
 N\setminus k    &   0 &1&   10 &   10^2 & 10^3 &10^4&10^5 \\ 
\hline&&&&&&&\\
11  &   8.62{\rm E}-06    &    8.66{\rm E}-06   &   4.05{\rm E}-05    &    3.66{\rm E}-05&   
1.02{\rm E}-06    &    1.65{\rm E}-08&    2.20{\rm E}-10 \\ 
   12  &   5.19{\rm E}-06    &    5.21{\rm E}-06   &   2.07{\rm E}-05    &    3.09{\rm E}-05&   
9.05{\rm E}-07    &    1.50{\rm E}-08&    2.01{\rm E}-10 \\ 
   23  &   1.31{\rm E}-07    &    1.31{\rm E}-07   &   1.34{\rm E}-07    &    6.06{\rm E}-06&   
3.56{\rm E}-07    &    7.13{\rm E}-09&    1.02{\rm E}-10 \\ 
   24  &   1.01{\rm E}-07    &    1.01{\rm E}-07   &   1.04{\rm E}-07    &    5.30{\rm E}-06&   
3.32{\rm E}-07    &    6.78{\rm E}-09&    9.74{\rm E}-11 \\ 
   47  &   1.52{\rm E}-09    &    1.52{\rm E}-09   &   1.55{\rm E}-09    &    2.96{\rm E}-07&   
9.38{\rm E}-08    &    2.87{\rm E}-09&    4.68{\rm E}-11 \\ 
   48  &   1.30{\rm E}-09    &    1.30{\rm E}-09   &   1.33{\rm E}-09    &    2.65{\rm E}-07&   
8.95{\rm E}-08    &    2.78{\rm E}-09&    4.57{\rm E}-11 \\ 
   95  &   1.64{\rm E}-11    &    1.64{\rm E}-11   &   1.67{\rm E}-11    &    2.68{\rm E}-09&   
1.23{\rm E}-08    &    9.59{\rm E}-10&    2.04{\rm E}-11 \\ 
   96  &   1.65{\rm E}-11    &    1.65{\rm E}-11   &   1.67{\rm E}-11    &    1.76{\rm E}-09&   
1.18{\rm E}-08    &    9.41{\rm E}-10&    2.02{\rm E}-11 
\end{array}                                        
\]
\caption{\label{Tab:06:02}Errors of the quadrature rule for integral \eqref{eq:exp:06a} with 
$\alpha=-1$, $\beta=1/2$ (top) and $\beta=3/2$ (bottom)}
\end{table}

\begin{table}[p]\small
\[
\begin{array}{l|c |c|c |c|c |c |c }
 N\setminus k    &   0 &1&   10 &   10^2 & 10^3 &10^4&10^5 \\ 
\hline&&&&&&&\\
 11  &   2.13{\rm E}-02    &    2.15{\rm E}-02   &   5.10{\rm E}-02    &    2.96{\rm E}-03&   
8.24{\rm E}-05    &    1.31{\rm E}-05&    1.27{\rm E}-06 \\ 
   12  &   5.71{\rm E}-02    &    5.72{\rm E}-02   &   1.21{\rm E}-01    &    1.63{\rm E}-03&   
5.31{\rm E}-05    &    1.69{\rm E}-06&    5.45{\rm E}-08 \\ 
   23  &   5.85{\rm E}-03    &    5.86{\rm E}-03   &   6.49{\rm E}-03    &    2.10{\rm E}-03&   
2.81{\rm E}-05    &    4.38{\rm E}-06&    3.79{\rm E}-07 \\ 
   24  &   2.15{\rm E}-02    &    2.15{\rm E}-02   &   2.49{\rm E}-02    &    1.66{\rm E}-03&   
5.36{\rm E}-05    &    1.71{\rm E}-06&    5.46{\rm E}-08 \\ 
   47  &   1.83{\rm E}-03    &    1.83{\rm E}-03   &   1.78{\rm E}-03    &    1.87{\rm E}-03&   
4.51{\rm E}-05    &    2.24{\rm E}-06&    1.28{\rm E}-07 \\ 
   48  &   7.73{\rm E}-03    &    7.73{\rm E}-03   &   8.01{\rm E}-03    &    1.59{\rm E}-03&   
5.39{\rm E}-05    &    1.72{\rm E}-06&    5.47{\rm E}-08 \\ 
   95  &   6.11{\rm E}-04    &    6.11{\rm E}-04   &   5.92{\rm E}-04    &    1.98{\rm E}-03&   
5.17{\rm E}-05    &    1.84{\rm E}-06&    6.65{\rm E}-08 \\ 
   96  &   2.75{\rm E}-03    &    2.75{\rm E}-03   &   2.77{\rm E}-03    &    6.56{\rm E}-03&   
5.42{\rm E}-05    &    1.72{\rm E}-06&    5.48{\rm E}-08 
\end{array}                                         
\]
\[
\begin{array}{l|c |c|c |c|c |c |c }
 N\setminus k    &   0 &1&   10 &   10^2 & 10^3 &10^4&10^5 \\ 
\hline&&&&&&&\\
 11  &   2.08{\rm E}-03    &    2.14{\rm E}-03   &   6.62{\rm E}-03    &    2.23{\rm E}-04&   
2.01{\rm E}-05    &    2.02{\rm E}-06&    2.02{\rm E}-07 \\ 
   12  &   1.39{\rm E}-03    &    1.40{\rm E}-03   &   5.83{\rm E}-03    &    3.02{\rm E}-05&   
2.47{\rm E}-07    &    2.47{\rm E}-09&    2.65{\rm E}-11 \\ 
   23  &   2.20{\rm E}-04    &    2.21{\rm E}-04   &   2.96{\rm E}-04    &    5.41{\rm E}-05&   
2.79{\rm E}-06    &    2.88{\rm E}-07&    2.88{\rm E}-08 \\ 
   24  &   2.87{\rm E}-04    &    2.87{\rm E}-04   &   4.17{\rm E}-04    &    2.47{\rm E}-05&   
1.02{\rm E}-07    &    8.26{\rm E}-10&    1.00{\rm E}-11 \\ 
   47  &   2.88{\rm E}-05    &    2.89{\rm E}-05   &   2.91{\rm E}-05    &    3.11{\rm E}-05&   
3.53{\rm E}-07    &    4.29{\rm E}-08&    4.28{\rm E}-09 \\ 
   48  &   5.32{\rm E}-05    &    5.32{\rm E}-05   &   5.87{\rm E}-05    &    2.34{\rm E}-05&   
7.33{\rm E}-08    &    1.45{\rm E}-10&    3.15{\rm E}-12 \\ 
   95  &   4.40{\rm E}-06    &    4.40{\rm E}-06   &   4.17{\rm E}-06    &    2.90{\rm E}-05&   
3.98{\rm E}-08    &    6.83{\rm E}-09&    6.79{\rm E}-10 \\ 
   96  &   9.51{\rm E}-06    &    9.51{\rm E}-06   &   9.76{\rm E}-06    &    5.08{\rm E}-05&   
7.70{\rm E}-08    &    1.22{\rm E}-10&    5.94{\rm E}-13 
\end{array}                                        
\]
\caption{\label{Tab:06:03}Errors of the quadrature rule for integral \eqref{eq:exp:06b} with 
$\alpha=0$, $\beta=1/2$ (top) and $\beta=3/2$ (bottom)}

 \[
\begin{array}{l|c |c|c |c|c |c |c } 
  N\setminus k    &   0 &1&   10 &   10^2 & 10^3 &10^4&10^5 \\ 
 \hline&&&&&&&\\                
 11  &   1.74{\rm E}-02    &    1.75{\rm E}-02   &   2.53{\rm E}-02    &    1.63{\rm E}-03&   
5.63{\rm E}-05    &    1.73{\rm E}-06&    5.49{\rm E}-08 \\ 
   12  &   6.28{\rm E}-02    &    6.29{\rm E}-02   &   5.65{\rm E}-02    &    1.41{\rm E}-03&   
6.94{\rm E}-05    &    1.67{\rm E}-06&    5.46{\rm E}-08 \\ 
   23  &   5.34{\rm E}-03    &    5.34{\rm E}-03   &   6.02{\rm E}-03    &    1.92{\rm E}-03&   
5.60{\rm E}-05    &    1.73{\rm E}-06&    5.49{\rm E}-08 \\ 
   24  &   2.20{\rm E}-02    &    2.20{\rm E}-02   &   2.26{\rm E}-02    &    2.79{\rm E}-03&   
6.28{\rm E}-05    &    1.68{\rm E}-06&    5.47{\rm E}-08 \\ 
   47  &   1.76{\rm E}-03    &    1.76{\rm E}-03   &   1.81{\rm E}-03    &    1.59{\rm E}-03&   
5.40{\rm E}-05    &    1.73{\rm E}-06&    5.49{\rm E}-08 \\ 
   48  &   7.78{\rm E}-03    &    7.78{\rm E}-03   &   7.83{\rm E}-03    &    8.18{\rm E}-04&   
5.07{\rm E}-05    &    1.68{\rm E}-06&    5.47{\rm E}-08 \\ 
   95  &   6.01{\rm E}-04    &    6.01{\rm E}-04   &   6.06{\rm E}-04    &    1.30{\rm E}-03&   
5.68{\rm E}-05    &    1.72{\rm E}-06&    5.49{\rm E}-08 \\ 
   96  &   2.75{\rm E}-03    &    2.75{\rm E}-03   &   2.75{\rm E}-03    &    3.19{\rm E}-03&   
6.43{\rm E}-05    &    1.66{\rm E}-06&    5.47{\rm E}-08 
 \end{array}                                          
 \] 
 \[ 
 \begin{array}{l|c |c|c |c|c |c |c } 
  N\setminus k    &   0 &1&   10 &   10^2 & 10^3 &10^4&10^5 \\ 
 \hline&&&&&&&\\
 
   11  &   1.10{\rm E}-03    &    1.12{\rm E}-03   &   2.36{\rm E}-03    &    1.59{\rm E}-05&   
3.03{\rm E}-07    &    3.58{\rm E}-09&    4.59{\rm E}-11 \\ 
   12  &   1.79{\rm E}-03    &    1.80{\rm E}-03   &   1.97{\rm E}-03    &    7.04{\rm E}-05&   
1.36{\rm E}-06    &    1.66{\rm E}-08&    2.09{\rm E}-10 \\ 
   23  &   1.54{\rm E}-04    &    1.55{\rm E}-04   &   2.07{\rm E}-04    &    3.87{\rm E}-05&   
1.95{\rm E}-07    &    1.97{\rm E}-09&    2.56{\rm E}-11 \\ 
   24  &   3.08{\rm E}-04    &    3.09{\rm E}-04   &   3.45{\rm E}-04    &    6.30{\rm E}-05&   
5.34{\rm E}-07    &    6.14{\rm E}-09&    7.78{\rm E}-11 \\ 
   47  &   2.44{\rm E}-05    &    2.44{\rm E}-05   &   2.64{\rm E}-05    &    2.15{\rm E}-05&   
7.43{\rm E}-08    &    7.59{\rm E}-10&    1.05{\rm E}-11 \\ 
   48  &   5.41{\rm E}-05    &    5.42{\rm E}-05   &   5.57{\rm E}-05    &    1.20{\rm E}-05&   
1.54{\rm E}-07    &    2.14{\rm E}-09&    2.79{\rm E}-11 \\ 
   95  &   4.07{\rm E}-06    &    4.08{\rm E}-06   &   4.16{\rm E}-06    &    1.75{\rm E}-05&   
1.16{\rm E}-07    &    2.00{\rm E}-10&    3.91{\rm E}-12 \\ 
   96  &   9.56{\rm E}-06    &    9.56{\rm E}-06   &   9.62{\rm E}-06    &    1.90{\rm E}-05&   
1.68{\rm E}-07    &    6.33{\rm E}-10&    9.81{\rm E}-12 
\end{array}                  
 \]
 \caption{\label{Tab:06:04}Errors of the quadrature rule for integral \eqref{eq:exp:06b} with 
 $\alpha=-1$, $\beta=1/2$ (top) and $\beta=3/2$ (bottom)}
 \end{table}

\appendix

\section{Some relevant properties for Chebyshev polynomials}

For the sake of completeness we present in this section those properties of
Chebyshev polynomials we have used in this work.  These results  can be
found in many classical text books on special
functions or Chebyshev polynomials (see for instance \cite[Ch. 22]{AbrSt} or 
\cite{Riv:1990}).

From the definitions of the Chebyshev polynomials of first and second kind we
have the relations
\begin{equation}
 \label{eq:def:Tn:Un}
 T_n(\cos\theta)=\cos n\theta,\qquad U_{n}(\cos\theta)=
\frac{1}{n+1}T_{{n +1}}'(\cos\theta)= \frac{\sin
(n+1)\theta}{\sin\theta},
\end{equation}
As a byproduct, one can deduce that if $n$ is even (respectively odd), so are 
$T_n$
and $U_n$. Note that as usual in this work, we have taken $U_{-1}=0$,
which is also consistent with   \eqref{eq:def:Tn:Un}. 
Both families of  polynomials obey the  recurrence
relation
\begin{equation}
\label{eq:TnRecc}
P_{n+1}(x) = 2x P_{n}(x)-P_{n-1}(x)
\end{equation}
but with, obviously, different starting values, simply $T_0(x)= 1$, $T_1(x)=x$
and
$U_{-1}(x)= 0$, $U_0(x)=1$ respectively. 


From \eqref{eq:def:Tn:Un} we easily deduce 
\[
T_n'(\cos\theta)=n \frac{\sin
n\theta}{\sin\theta},\qquad \sin\theta \,T_n''(\cos\theta)=-n\frac{\rm d}{\rm
d\theta}\bigg(\frac{\sin n\theta}{\sin \theta}\bigg).
\] 
Therefore, 
\begin{equation}
\label{eq:Tnbounded}
\|T_n\|_{{L_\infty(-1,1)}} \le T_n(1)= 1,\qquad
\|U_n\|_{{L_\infty(-1,1)}}=\frac{1}{n+1}\|T_{n+1}'\|_{L_\infty(-1,1)}=n+1
\end{equation}
and  (recall that $w(x)=\sqrt{1-x^2} $)
\begin{equation}
\label{eq:Tnbounded:02}
\|wT'_{n+1}\|_{{L_\infty(-1,1)}}=(n+1),\quad
\|w
T''_{n+1}\|_{{L_\infty(-1,1)}}\le
 C (n+1)^3
\end{equation}
where $C$ is independent of $n$.

Unlike $T_n$, $U_n$ is not uniformly bounded in $n$ and $x\in [-1,1]$. However, 
\begin{equation}
\label{eq:UnIntegral}
\|U_n\|_{1,\omega}^2=\int_{-1}^1 |U_n(x)|^2\,\sqrt{1-x^2}\,{\rm
d}x=\int_0^\pi \sin^2n \theta\,\mathrm{d}\theta=  \frac{\pi}{2} . 
\end{equation}

On the other hand, 
\begin{equation}
\label{eq:intTn}
 \int_{-1}^1 T_{n}(x)\,{\rm d}x=\int_0^\pi \cos n\theta \sin\theta\,{\rm
d}\theta=\left\{\begin{array}{ll}
-\frac{2}{n^2-1},\quad& \text{if $n$ is even},\\
0,& \text{otherwise}.
 \end{array}\right.
\end{equation}
The trigonometric identity
\[
\cos n\theta\  {\sin
(m+1)\theta} =\frac{1}{2}\big( {\sin(m+n+1)\theta}
+ {\sin(m+1-n)\theta} \big)
\]
implies 
\begin{equation}
 \label{eq:ProdTnUm}
T_n U_m=\left\{\begin{array}{ll} 
                  \frac12\displaystyle \big(U_{ m+n }+U_{  m-n } \big),\quad
&\text{if $m\ge n-1$},\\[1.5ex]
  \frac12\displaystyle \big(U_{m+n}-U_{n-m-2} \big),\quad
&\text{if $m\le n-2$}.
                \end{array}
\right.
\end{equation}
In particular, we obtain for $n\ge 1$
\begin{eqnarray}
 2x T_n'(x)&=&2  n T_1(x)
U_{n-1}(x)=n\big[U_n(x)+U_{n-2}(x)\big], \\
 T_n(x)&=&T_n(x)U_0(x)=\frac{1}{2}\big[U_n(x)-U_{n-2}(x)\big].\label{eq:UnTn:02}
\end{eqnarray}  

Finally, it holds
\begin{equation}
 \frac{T_n(x)-T_{n}(y)}{x-y}=
2\sum_{j=0}^{n-2} U_j(x) T_{n-1-j}(y)+U_{n-1}(x)=
2\sum_{j=0}^{n-2} U_j(y) T_{n-1-j}(x)+U_{n-1}(y)\quad
\label{eq:Tn-Tn}
\end{equation}
which can be easily proven by induction on $n$.

\subsection*{Acknowledgements}

The author is supported partially by Project  MTM2010-21037. The author wants 
to thank Prof. Ivan Graham for several useful discussions which help to improve 
both the quality and readability of this paper. 
\bibliography{biblio.bib}

\end{document}